\documentclass[review]{elsarticle}

\usepackage{mathrsfs,amsmath}
\usepackage{graphicx}
\usepackage{subfigure}
\usepackage{amssymb}
\usepackage{graphicx}
\usepackage{amsfonts}
\usepackage{amsthm}
\usepackage{enumerate}
\usepackage{color}
\usepackage{bm}
\usepackage{float}
\usepackage{caption}
\usepackage{epstopdf}
\usepackage{booktabs}
\usepackage{dutchcal}
\usepackage{epsfig}
\usepackage[colorlinks,linkcolor=red,anchorcolor=blue,citecolor=green]{hyperref}
\biboptions{numbers,sort&compress}
\newtheorem{theorem}{Theorem}
\newtheorem{definition}{Definition}

\newtheorem{lemma}{Lemma}
\newtheorem{property}{Property}
\newtheorem{corollary}{Corollary}

\usepackage{nomencl} 
\makenomenclature 

\begin{document}

\begin{frontmatter}



\title{Polar Linear Canonical Wavelet Transform: Theory and Its Application}


\author{Hui Zhao$^{a,b}$}
\author{Bing-Zhao Li$^{a,b}$\corref{mycorrespondingauthor}}
\cortext[mycorrespondingauthor]{Corresponding author}\ead{li\_bingzhao@bit.edu.cn}

\address{$^{a}$School of Mathematics and Statistics, Beijing Institute of Technology, Beijing 100081, China}
\address{$^{b}$Beijing Key Laboratory on MCAACI, Beijing Institute of Technology, Beijing 100081, China}

\begin{abstract}
 The polar wavelet transform (PWT) has been proven to be a powerful mathematical tool for signal and image processing in recent years. Due to the increasing demand for directional representations of signals in engineering, it is impossible to fully exploit the intrinsic directional features of signals to describe high-dimensional signals like images. Focusing on this problem, the polar linear canonical wavelet transform (PLCWT) is proposed in this paper. Firstly, the theory of the PLCWT is investigated in detail, including its definition, basic properties and inversion formula. Secondly, the convolution and correlation theorems of the PLCWT are derived. Thirdly, uncertainty principles related to the PLCWT are obtained. Finally, the potential application of the PLCWT in image edge detection is discussed. Simulation results verify the correctness and effectiveness of the proposed method. 
\end{abstract}

%

\end{frontmatter}


\section{Introduction}
Wavelet transform (WT) is a powerful mathematical analysis tool that has been widely used in many fields, including radar target recognition, signal analysis, image processing and medicine, etc \cite{ref1,ref2,ref3,ref4,ref5,ref6,ref7,ref8,ref9,ref10}. To effectively analyze and characterize signals, researchers have proposed a series of mathematical transform methods, such as fractional Fourier transform (FRFT) \cite{ref29,ref30,ref31,ref32}, short-time fractional Fourier transform (STFRFT) \cite{ref33,ref3444}, linear canonical transform (LCT) \cite{ref34,ref35,ref36,ref37,ref38,ref39}, fractional wavelet transform (FRWT) \cite{ref13,ref14,ref15,ref16,ref166}, linear canonical wavelet transform (LCWT) \cite{ref17,ref18,ref19,ref20}, fractional Stockwell transform (FRST) \cite{ref201,ref202,ref203}, linear canonical Stockwell transform (LCST) \cite{ref204}, etc. Due to the increasing requirements for the direction feature representation of signals like images in image processing, the above methods cannot optimally analyze the directionality of high-dimensional signals. They can only simply describe the directionality of an image and are limited in describing the rotation invariance of an image. This makes it difficult to accurately analyze and process images in pre-selected directions.

Based on the problem mentioned above, a number of mathematical tools have been proposed, such as the polar wavelet transform (PWT) \cite{ref22,ref23}, polar Stockwell transform (PST) \cite{ref11,ref12} and polar offset linear canonical transform \cite{ref21}, etc. These tools have been widely used in image compression, texture analysis, edge detection and other fields \cite{ref24,ref25,ref26}. With the continuous development of information processing technology, high-dimensional signals have become more and more complex in practical applications, and most of them are non-stationary signals. However, the traditional high-dimensional mathematical analysis tools can not achieve ideal results in the processing and analysis of these signals, so finding mathematical methods to deal with such signals has become a research hotspot.

In this paper, a new transform method is proposed, namely the polar linear canonical wavelet transform (PLCWT), which combines the advantages of the PWT and the LCT. The PLCWT not only has high directional analysis, which improves the efficiency and performance of capturing local information along different angles, but also has the ability to characterize two-dimensional signals on the time-LCT domain plane. The main contributions of this study are summarized as follows.
\begin{enumerate}[(1)]
	\item To better analyze images in the combined domain of time, direction and the LCT, a novel PLCWT is proposed, and its basic properties are investigated.
	\item The convolution and correlation theorems are obtained based on the convolution and correlation operators. And the uncertainty principles related to the PLCWT are proved.
	\item The potential application of the PLCWT in image edge detection is discussed. Simulation results verify the correctness and effectiveness of the proposed transform method.
\end{enumerate}

The remainder of this study is structured as follows: Section \ref{Preli} provides some basic preliminaries. In Section \ref{PLCWT}, the theory of the PLCWT is investigated, including its definition, basic properties and inversion formula. In Section \ref{Convolution}, we present convolution and correlation theorems for the PLCWT. In Section \ref{Uncertainty}, the three uncertainty principles of the PLCWT are proved. In Section \ref{Applications}, a potential application is presented to demonstrate the importance of the PLCWT. Finally, Section \ref{Con} concludes the paper.

\section{Preliminaries}
\label{Preli}
\subsection{Polar Wavelet Transform}
In this subsection, we provide some necessary background and notations of the PWT.\

The polar wavelet (rotational wavelet) is a family of functions constructed from a mother wavelet $\psi\in L^{2}(\mathbb{R}^{2})$ by the combined action of translation, dilation and rotations as \cite{ref22,ref26}
\begin{align}
	\begin{split}
		\psi_{\boldsymbol{b},a,\theta}(\boldsymbol{t})=\frac{1}{a}\psi\left(\frac{R_{-\theta}\left(\boldsymbol{t}-\boldsymbol{b} \right) }{a} \right), 
	\end{split}
	\label{1}
\end{align}
where $a\in{\mathbb{R}^{+}},\,\boldsymbol{b}\in{\mathbb{R}^{2}}, \,\theta\in\left[0,2\pi \right]$, and $R_{\theta}$ stands for rotation parameter give by \cite{ref22,ref26}
\begin{align}
	\begin{split}
		R_{\theta}(\boldsymbol{t})=
		\begin{gathered}
			\begin{pmatrix} \cos\theta & \sin\theta \\ -\sin\theta & \cos\theta \end{pmatrix}
			\begin{pmatrix} t_{1} \\ t_{2}  \end{pmatrix}
		\end{gathered},\quad \forall\,\boldsymbol{t}=(t_{1},t_{2})\in{\mathbb{R}^{2}}. 
	\end{split}
	\label{2}
\end{align}

A polar wavelet is a function $\psi\in L^{2}(\mathbb{R}^{2})$ which satisfies the condition \cite{ref22,ref26}
\begin{align}
	\begin{split}
		C_{\psi}=\int_{0}^{\infty}\int_{0}^{2\pi}\bigg| \mathscr{F}[\psi](aR_{-\theta}\boldsymbol{\xi})\bigg|^{2}\frac{{\rm{d}}a\rm{d}\theta}{a}<\infty,
	\end{split}
	\label{3}
\end{align}
where $\mathscr{F}$ is the polar Fourier transform of $\psi$. Condition (\ref{3}) is called the admissibility condition which guarantees the existence of the inversion formula for the continuous polar wavelet transform \cite{ref22,ref26}.

\begin{definition}[PWT \cite{ref22,ref26}]\label{1}Let $\psi\in L^{2}(\mathbb{R}^{2})$ be an admissible polar wavelet, the polar wavelet transform (PWT) of any $f\in L^{2}(\mathbb{R}^{2})$ to $\psi$ is denoted by 
\begin{align}
	\begin{split}
		W_{f}(\boldsymbol{b},a,\theta)=\frac{1}{a}\int_{\mathbb{R}^{2}}f(\boldsymbol{t})\overline{\psi\left(\frac{R_{-\theta}\left(\boldsymbol{t}-\boldsymbol{b} \right) }{a} \right)}{\rm{d}}\boldsymbol{t}.
	\end{split}
	\label{4}
\end{align}
	
It can also be defined as a conventional convolution
\begin{align}
	\begin{split}
		W_{f}(\boldsymbol{b},a,\theta)&=f(\boldsymbol{t})\ast\left(a^{-1}\overline{\psi (R_{-\theta}(-\boldsymbol{t}/a))} \right)\\
		&=\left\langle f,\psi_{\boldsymbol{b},a,\theta} \right\rangle_{L^{2}(\mathbb{R}^{2})},
	\end{split}
	\label{5}
\end{align}
where $\ast$ and $\left\langle\cdot,\cdot \right\rangle$ denote the classical convolution operator and the inner product, respectively.
\end{definition}

\subsection{Linear canonical transform}
\begin{definition}[LCT \cite{ref34,ref35,ref36,ref37}]\label{2}The linear canonical transform (LCT) with parameters $M=(A,B;C,D)$ of a signal $f(\boldsymbol{t})\in L^{2}(\mathbb{R}^{2})$ is defined by 
\begin{align}
	\begin{split}
		F^{M}(\boldsymbol{\xi})=\mathscr{L}^{M}[f(\boldsymbol{t})](\boldsymbol{\xi})=\begin{cases}
			\int_{\mathbb{R}^{2}}f(\boldsymbol{t})h_{M}(\boldsymbol{t},\boldsymbol{\xi}){\rm{d}}\boldsymbol{t},   &B\neq0  \\
			\sqrt{D}e^{i \frac{CD\boldsymbol{\xi}^{2}}{2}}\delta(\boldsymbol{t}-{\rm{d}}\boldsymbol{\xi}),  &B=0
		\end{cases}
		\label{6}
	\end{split}
\end{align}
where the kernel $h_{M}(\boldsymbol{t},\boldsymbol{\xi})$ is
\begin{align}	
	h_{M}(\boldsymbol{t},\boldsymbol{\xi})=\dfrac{1}{2\pi B}e^{j\left[ \frac{A}{2B}|\boldsymbol{t}|^{2}-\frac{1}{B}\boldsymbol{t}\cdot\boldsymbol{\xi}+\frac{D}{2B}|\boldsymbol{\xi}|^{2}\right] }.
	\label{7}
\end{align}
Here, $\mathscr{L}^{M}$ denotes the LCT operator, $A$, $B$, $C$, and $D$ are real numbers satisfying $AD-BC=1$.
\end{definition}
	
The convolution and product theorems of the LCT are introduced in \cite{ref40}
\begin{align}	
	f(\boldsymbol{t})\,\Theta_{M}\, g(\boldsymbol{t})=e^{-j \frac{A}{2B}\boldsymbol{t}^{2}}\left[ \left(f(\boldsymbol{t})e^{j\frac{A}{2B}\boldsymbol{t}^{2}} \right)\ast g(\boldsymbol{t}) \right], 
	\label{8}
\end{align}
\textcolor{red}{where $\Theta_{M}$ indicates the linear canonical convolution operator}, and $\ast$ denotes the conventional convolution operator for the FT. Then,
the convolution theorem of the LCT for the signal $f(\boldsymbol{t})$ and $g(\boldsymbol{t})$ is given by \cite{ref40}
\textcolor{red}{\begin{align}	
	f(\boldsymbol{t})\,\Theta_{M}\, g(\boldsymbol{t})\stackrel{\mathscr{L}^{M}}{\longleftrightarrow}F^{M}(\boldsymbol{\xi})G(\frac{\boldsymbol{\xi}}{B}), 
	\label{9}
\end{align}}
where $F^{M}(\boldsymbol{\xi})$ and $G(\boldsymbol{\xi})$ denote the LCT of $f(\boldsymbol{t})$ and the FT of $f(\boldsymbol{t})$, respectively.

\section{Polar Linear Canonical Wavelet Transform }
\label{PLCWT}
In this section, a new definition of the PLCWT is proposed, which addresses the limitations of the PWT and the LCT. Then, some basic properties are derived. More importantly, it has shown a high degree of directionality compared to previusly defined LCWT.

\subsection{Definition}
\label{De}
In this subsection, a polar linear canonical wavelet transform (PLCWT) is defined based on the convolution operation in the LCT domain. 
\begin{definition}[PLCWT]\label{3}The PLCWT of a signal $f(\boldsymbol{t})\in L^{2}(\mathbb{R}^{2})$ with parameters $M=(A,B;C,D)$ and polar mother wavelet $\psi\in L^{2}(\mathbb{R}^{2})$ is defined as
\begin{align}
	\begin{split}
		W^{M}_{f}(\boldsymbol{b},a,\theta)&=f(\boldsymbol{t})\Theta_{M}\left(a^{-1}\overline{\psi (R_{-\theta}(-\boldsymbol{t}/a))} \right)\\
		&=e^{-j \frac{A}{2B}\boldsymbol{b}^{2}}\left\langle f(\cdot)e^{-j \frac{A}{2B}(\cdot)^{2}},a^{-1}\psi_{\boldsymbol{b},a,\theta}(\cdot) \right\rangle_{L^{2}(\mathbb{R}^{2})}\\
		&=\int_{\mathbb{R}^{2}}f(\boldsymbol{t})\overline{\psi_{\boldsymbol{b},a,\theta}^{M}(\boldsymbol{t})}{\rm{d}}\boldsymbol{t}.
	\end{split}
	\label{10}
\end{align}
where the transform kernel $\psi_{\boldsymbol{b},a,\theta}^{M}$ satisfies
\begin{align}
	\begin{split}
		\psi_{\boldsymbol{b},a,\theta}^{M}(\boldsymbol{t})=e^{-j\frac{A}{2B}(\boldsymbol{t}^{2}-\boldsymbol{b}^{2})}\psi_{\boldsymbol{b},a,\theta}(\boldsymbol{t}),
	\end{split}
	\label{11}
\end{align}
where $\psi_{\boldsymbol{b},a,\theta}(\boldsymbol{t})$ is give by (\ref{1}), $a\in{\mathbb{R}^{+}},\,\boldsymbol{b}\in{\mathbb{R}^{2}}, \,\theta\in\left[0,2\pi \right]$. When $M=(0,1;-1,0)$, the PLCWT reduced to the PWT.\
\end{definition}
	A polar mother wavelet $\psi\in L^{2}(\mathbb{R}^{2})$ related to the LCT is known as the admissibility condition if 
	\begin{align}
		\begin{split}
			C_{\psi,M}=\!\!\int_{0}^{\infty}\!\!\int_{0}^{2\pi}\bigg|\mathscr{L}^{M}\left\lbrace e^{-j \frac{A}{2B}(\cdot)^{2}}\psi \right\rbrace (aR_{-\theta}\boldsymbol{\xi})\bigg|^{2}\frac{{\rm{d}}a\rm{d}\theta}{a}<\infty.
		\end{split}
		\label{12}
	\end{align}	
	The identity (\ref{12}) is called the admissibility condition which guarantees the existence of the inversion formula for the PLCWT.\
	
	The PLCWT of a signal $f(\boldsymbol{t})$ in (\ref{10}) can be rewritten as
	\begin{align}
		\begin{split}
			W^{M}_{f}(\boldsymbol{b},a,\theta)&=e^{-j\frac{A}{2B}\boldsymbol{b}^{2}}\int_{\mathbb{R}^{2}}\left( f(\boldsymbol{t})e^{j\frac{A}{2B}\boldsymbol{t}^{2}}\right)\overline{\psi_{\boldsymbol{b},a,\theta}(\boldsymbol{t})}{\rm{d}}\boldsymbol{t}.
		\end{split}
		\label{13}
	\end{align}
	The computation of the PLCWT corresponds to the following steps:
	\begin{enumerate}
		\item A product by a chirp signal, i.e. $f(\boldsymbol{t})\rightarrow  \widetilde{f}(\boldsymbol{t})=f(\boldsymbol{t})\cdot e^{j\frac{A}{2B}\boldsymbol{t}^{2}}$;
		\item A traditional PWT, i.e. $\widetilde{f}(\boldsymbol{t})\rightarrow W_{\widetilde{f}}\,(\boldsymbol{b},a,\theta)$;
		\item Another product by a chip signal, i.e. 
		\begin{align}
			\begin{split}
				W_{\widetilde{f}}\,(\boldsymbol{b},a,\theta)\rightarrow W^{M}_{f}\,(\boldsymbol{b},a,\theta)=W_{\widetilde{f}}\,(\boldsymbol{b},a,\theta)e^{-j\frac{A}{2B}\boldsymbol{b}^{2}}.\nonumber
			\end{split}
		\end{align}
	\end{enumerate}
	
	It can be seen from the above steps that the PLCWT can be realized through the discrete algorithm of the PWT. This also means that the computational complexity of the PLCWT depends on the computational complexity of the PWT, and the implementation time of the PWT is $O(N)$. Therefore, the computational complexity of the PLCWT is $O(N)$. As can be seen from the above, the PLCWT is easy to implement in practice and has low complexity. Furthermore, the proposed PLCWT can be decomposed around the PWT embedded in its definition, as shown in Fig. \ref{fig:1}.\
	\begin{figure}[t!]
		\centering
		\includegraphics[width=1.0\linewidth]{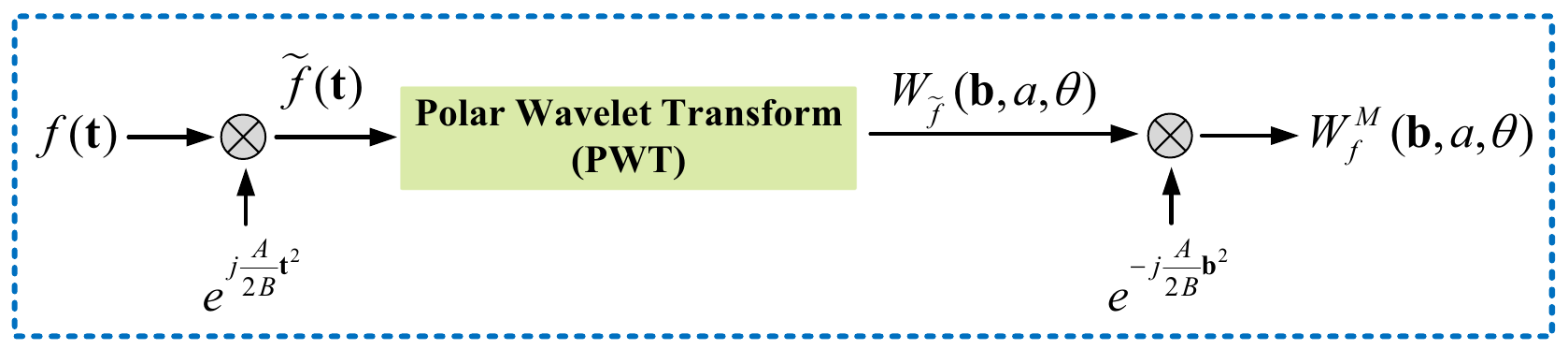}\\
		\caption{Structure of computing the proposed PLCWT.}
		\label{fig:1} 
	\end{figure}
	\subsection{Basic Propreties}
	\label{Ba}
	In this subsection, we investigate some basic properties of the PLCWT.\
	\begin{property}[Linearity]\label{property1}
		If $\psi\in L^{2}(\mathbb{R}^{2})$ is polar mother wavelet, $f\in L^{2}(\mathbb{R}^{2})$, $g\in L^{2}(\mathbb{R}^{2})$, and $h(\boldsymbol{t})=\alpha f(\boldsymbol{t})+\beta g(\boldsymbol{t})$, then 
		\begin{align}
			\begin{split}
				W^{M}_{h}(\boldsymbol{b},a,\theta)=\alpha W^{M}_{f}(\boldsymbol{b},a,\theta)+\beta W^{M}_{g}(\boldsymbol{b},a,\theta),  
			\end{split}
			\label{14}
		\end{align}
		where $\alpha,\beta\in{\mathbb{C}}$.
	\end{property}
	\begin{property}[Scaling]\label{property2}
		If $\psi\in L^{2}(\mathbb{R}^{2})$ is polar mother wavelet, $f\in L^{2}(\mathbb{R}^{2})$, and $h(\boldsymbol{t})=f(c\boldsymbol{t})$, then 
		\begin{align}
			\begin{split}
				W^{M}_{h}(\boldsymbol{b},a,\theta)=\frac{1}{c} W^{M_{1}}_{f}(\boldsymbol{b}c,ac,\theta),  
			\end{split}
			\label{15}
		\end{align}
		where $M_{1}=(A_{1},B_{1};C_{1},D_{1})$ satisfies $\frac{A_{1}}{B_{1}}=\frac{A}{BC^{2}}$, $c\in{\mathbb{C}}$, and $c \neq0$.
	\end{property}
	\begin{proof}
		Form the PLCWT, we have
		\begin{align}
			\begin{split}
				&\quad W^{M}_{h}(\boldsymbol{b},a,\theta)=\left\langle f(c\boldsymbol{t}),\psi^{M}_{\boldsymbol{b},a,\theta} \right\rangle_{L^{2}(\mathbb{R}^{2})}\\
				&=\frac{1}{a}\int_{\mathbb{R}^{2}}f(c\boldsymbol{t})\overline{\psi\left(\frac{R_{-\theta}\left(\boldsymbol{t}-\boldsymbol{b} \right) }{a} \right)}e^{j\frac{A}{2B}(\boldsymbol{t}^{2}-\boldsymbol{b}^{2})}{\rm{d}}\boldsymbol{t}\\
				&=\frac{1}{ac}\int_{\mathbb{R}^{2}}f(\boldsymbol{z})\overline{\psi\left(\frac{R_{-\theta}\left(\dfrac{\boldsymbol{z}}{c}-\boldsymbol{b} \right) }{a} \right)}e^{j\frac{A}{2B}\left(  \frac{\boldsymbol{t}^{2}-c^{2}\boldsymbol{b}^{2}}{c^{2}} \right)  }{\rm{d}}\boldsymbol{z}\\
				&=\frac{1}{c}\left\lbrace \frac{1}{a}\int_{\mathbb{R}^{2}}f(\boldsymbol{z})\overline{\psi\left(\frac{R_{-\theta}\left(\boldsymbol{z}-\boldsymbol{b}c \right) }{ac} \right)}e^{j\frac{A}{2Bc^{2}}\left[ \boldsymbol{t}^{2}-(\boldsymbol{b}c)^{2}\right]}\right\rbrace {\rm{d}}\boldsymbol{z}\\
				&=\frac{1}{c}W^{M_{1}}_{f}(\boldsymbol{b}c,ac,\theta),
			\end{split}
			\label{16}
		\end{align}
		where $M_{1}=(A_{1},B_{1};C_{1},D_{1})$ satisfies $\frac{A_{1}}{B_{1}}=\frac{A}{BC^{2}}$.
	\end{proof}
	\begin{property}[Translation]\label{property3}
		If $\psi\in L^{2}(\mathbb{R}^{2})$ is polar mother wavelet, $f\in L^{2}(\mathbb{R}^{2})$, and $h(\boldsymbol{t})=f(\boldsymbol{t}-\boldsymbol{c})e^{-j\frac{A}{B}\boldsymbol{t}\cdot\boldsymbol{c}}$, then 
		\begin{align}
			\begin{split}
				W^{M}_{h}(\boldsymbol{b},a,\theta)=e^{-j\frac{A}{B}\boldsymbol{b}\cdot\boldsymbol{c}}\,W^{M}_{f}(\boldsymbol{b}-\boldsymbol{c},a,\theta),  \quad \forall\, \boldsymbol{c}\in \mathbb{R}^{2}.
			\end{split}
			\label{17}
		\end{align}
	\end{property}
	\begin{proof}
		For $\boldsymbol{c}\in \mathbb{R}^{2}$, we get
		\begin{align}
			\begin{split}
				&\quad W^{M}_{h}(\boldsymbol{b},a,\theta)\\
				&=\left\langle f(\boldsymbol{t}-\boldsymbol{c})e^{-j\frac{A}{B}\boldsymbol{t}\cdot\boldsymbol{c}},\psi^{M}_{\boldsymbol{b},a,\theta} \right\rangle_{L^{2}(\mathbb{R}^{2})}\\
				&=\frac{1}{a}\int_{\mathbb{R}^{2}}f(\boldsymbol{t}-\boldsymbol{c})e^{-j\frac{A}{B}\boldsymbol{t}\cdot\boldsymbol{c}}\overline{\psi\left(\frac{R_{-\theta}\left(\boldsymbol{t}-\boldsymbol{b} \right) }{a} \right)}e^{j\frac{A}{2B}(\boldsymbol{t}^{2}-\boldsymbol{b}^{2})}{\rm{d}}\boldsymbol{t}\\
				&=\textcolor{red}{\frac{1}{a}e^{-j\frac{A}{B}\boldsymbol{b}\cdot\boldsymbol{c}}\int_{\mathbb{R}^{2}} f(\boldsymbol{z})\overline{\psi\left(\frac{R_{-\theta}\left(\boldsymbol{z}-\left(\boldsymbol{b}-\boldsymbol{c} \right) \right) }{a} \right)} e^{j\frac{A}{2B}\left[ \boldsymbol{z}^{2}-\left(\boldsymbol{b}-\boldsymbol{c} \right)^{2} \right]  }{\rm{d}}\boldsymbol{z}}\\
				&=e^{-j\frac{A}{B}\boldsymbol{b}\cdot\boldsymbol{c}}\,W^{M}_{f}(\boldsymbol{b}-\boldsymbol{c},a,\theta).
			\end{split}
			\label{18}
		\end{align}
		The proof is completed. 
	\end{proof}
	Next we prove some lemmas to study other fundamental properties.
	\begin{lemma}\label{Lemma2}
		If $\psi\in L^{2}(\mathbb{R}^{2})$ is polar mother wavelet. For any $f\in L^{2}(\mathbb{R}^{2})$, then the LCT of $\psi_{\boldsymbol{b},a,\theta}^{M}(\boldsymbol{t})$ is given by 
		\begin{align}
			\begin{split}
				\mathscr{L}^{M}\left\lbrace\psi_{\boldsymbol{b},a,\theta}^{M} \right\rbrace(\boldsymbol{\xi})=&ae^{j\left[  \frac{A}{2B}\boldsymbol{b}^{2}+\frac{D}{2B}\boldsymbol{\xi}^{2}-\frac{1}{B}\boldsymbol{b}\cdot\boldsymbol{\xi}-\frac{D}{2B}(a R_{-\theta}\boldsymbol{\xi})^{2}\right] }\\
				&\times\mathscr{L}^{M}\left\lbrace e^{-j \frac{A}{2B}(\cdot)^{2}}\psi \right\rbrace (aR_{-\theta}\boldsymbol{\xi}).
			\end{split}
			\label{21}
		\end{align}
	\end{lemma}
	\begin{proof}
		Base on the definition of the LCT, and let $\boldsymbol{w}=\boldsymbol{t}-\boldsymbol{b}$, we obtain 
		\begin{align}
			\begin{split}      \mathscr{L}^{M}\left\lbrace\psi_{\boldsymbol{b},a,\theta}^{M} \right\rbrace(\boldsymbol{\xi})
				&=\frac{1}{2\pi B}\int_{\mathbb{R}^{2}}\psi_{\boldsymbol{b},a,\theta}^{M}(\boldsymbol{t})e^{j\left( \frac{A}{2B}\boldsymbol{t}^{2}-\frac{1}{B}\boldsymbol{t}\cdot\boldsymbol{\xi}+\frac{D}{2B}\boldsymbol{\xi}^{2}\right) }{\rm{d}}\boldsymbol{t}\\
				&=\frac{1}{2\pi aB}\int_{\mathbb{R}^{2}}\psi\left(\frac{R_{-\theta}\boldsymbol{w} }{a} \right)e^{j\left[  \frac{A}{2B}\boldsymbol{b}^{2}-\frac{1}{B}(\boldsymbol{w}+\boldsymbol{b})\cdot\boldsymbol{\xi}+\frac{D}{2B}\boldsymbol{\xi}^{2}\right]}{\rm{d}}\boldsymbol{w}.
			\end{split}
			\label{22}
		\end{align}
		Using the change variables $\boldsymbol{z}=\frac{R_{-\theta}\,\boldsymbol{w}}{a}$, we get
		\begin{align}
			\begin{split}
			\mathscr{L}^{M}\left\lbrace\psi_{\boldsymbol{b},a,\theta}^{M} \right\rbrace(\boldsymbol{\xi})&=\frac{a}{2\pi B}\int_{\mathbb{R}^{2}}\psi(\boldsymbol{z})e^{j\left(\frac{A}{2B}\boldsymbol{b}^{2}-\frac{1}{B}aR_{\theta}\boldsymbol{z}\cdot\boldsymbol{\xi}-\frac{1}{B}\boldsymbol{b}\cdot\boldsymbol{\xi}+\frac{D}{2B}\boldsymbol{\xi}^{2} \right) }{\rm{d}}\boldsymbol{z}\\
				&=\frac{a}{2\pi B}\int_{\mathbb{R}^{2}}\psi(\boldsymbol{z})e^{j\left[ -\frac{1}{B}\boldsymbol{z}\cdot (aR_{-\theta}\boldsymbol{\xi})+\frac{D}{2B}(aR_{-\theta}\boldsymbol{\xi})^{2}\right]}{\rm{d}}\boldsymbol{z}\\
				&\quad \times e^{j\left[ \frac{A}{2B}\boldsymbol{b}^{2}-\frac{1}{B}\boldsymbol{b}\cdot\boldsymbol{\xi}+\frac{D}{2B}\boldsymbol{\xi}^{2}-\frac{D}{2B}(aR_{-\theta}\boldsymbol{\xi})^{2} \right]}\\
				&=ae^{j\left[  \frac{A}{2B}\boldsymbol{b}^{2}+\frac{D}{2B}\boldsymbol{\xi}^{2}-\frac{1}{B}\boldsymbol{b}\cdot\boldsymbol{\xi}-\frac{D}{2B}(a R_{-\theta}\boldsymbol{\xi})^{2}\right] }\\
				&\quad \times\mathscr{L}^{M}\left\lbrace e^{-j \frac{A}{2B}(\cdot)^{2}}\psi \right\rbrace (aR_{-\theta}\boldsymbol{\xi}).
			\end{split}
			\label{23}
		\end{align}
		The proof is completed. 
	\end{proof}
	\begin{lemma} \label{Lemma3}
		If $\psi\in L^{2}(\mathbb{R}^{2})$ is polar mother wavelet. For arbitrary $f\in L^{2}(\mathbb{R}^{2})$, we have
		\begin{align}
			\begin{split}
				W^{M}_{f}(\boldsymbol{b},a,\theta)&=a\!\int_{\mathbb{R}^{2}}\overline{e^{j\left[  \frac{A}{2B}\boldsymbol{b}^{2}+\frac{D}{2B}\boldsymbol{\xi}^{2}-\frac{1}{B}\boldsymbol{b}\boldsymbol{\xi}-\frac{D}{2B}(a R_{-\theta}\boldsymbol{\xi})^{2}\right] }}\\
				&\quad\times \overline{\mathscr{L}^{M}\!\left\lbrace e^{-j \frac{A}{2B}(\cdot)^{2}}\psi \right\rbrace\!(aR_{-\theta}\boldsymbol{\xi})}\mathscr{L}^{M}\!\left\lbrace f\right\rbrace(\boldsymbol{\xi}){\rm{d}}\boldsymbol{\xi}.
			\end{split}
			\label{24}
		\end{align}
	\end{lemma}
	\begin{proof}
		From LCT and Lemma \ref{Lemma2}, we obtain
		\begin{align}
			\begin{split}
				W^{M}_{f}(\boldsymbol{b},a,\theta)&=\left\langle f,\psi^{M}_{\boldsymbol{b},a,\theta} \right\rangle_{L^{2}(\mathbb{R}^{2})}\\
				&=\left\langle \mathscr{L}^{M}\left\lbrace f\right\rbrace ,\mathscr{L}^{M}\left\lbrace \psi^{M}_{\boldsymbol{b},a,\theta}\right\rbrace\right\rangle_{L^{2}(\mathbb{R}^{2})}\\
				&=\int_{\mathbb{R}^{2}}\mathscr{L}^{M}\left\lbrace f\right\rbrace (\boldsymbol{\xi})\overline{\mathscr{L}^{M}\left\lbrace \psi^{M}_{\boldsymbol{b},a,\theta}\right\rbrace(\boldsymbol{\xi})}{\rm{d}}\boldsymbol{\xi}\\
				&=a\!\int_{\mathbb{R}^{2}}\overline{e^{j\left[  \frac{A}{2B}\boldsymbol{b}^{2}+\frac{D}{2B}\boldsymbol{\xi}^{2}-\frac{1}{B}\boldsymbol{b}\boldsymbol{\xi}-\frac{D}{2B}(a R_{-\theta}\boldsymbol{\xi})^{2}\right] }}\\
				&\quad \times\overline{\mathscr{L}^{M}\!\left\lbrace e^{-j \frac{A}{2B}(\cdot)^{2}}\psi \right\rbrace\!(aR_{-\theta}\boldsymbol{\xi})}\mathscr{L}^{M}\!\left\lbrace f\right\rbrace(\boldsymbol{\xi}){\rm{d}}\boldsymbol{\xi}.
			\end{split}
			\label{25}
		\end{align}
		The proof is completed. 
	\end{proof}
	\begin{property}[Orthogonality Property]\label{property4}
		If $\psi\in L^{2}(\mathbb{R}^{2})$ is polar mother wavelet. Then, $\forall f, g\in L^{2}(\mathbb{R}^{2})$, we have
		\begin{align}
			\begin{split}
				\int_{0}^{2\pi}\!\int_{0}^{+\infty}\!\int_{\mathbb{R}^{2}}W^{M}_{f}(\boldsymbol{b},a,\theta)&\overline{W^{M}_{g}(\boldsymbol{b},a,\theta)}\frac{{\rm{d}}\boldsymbol{b}{\rm{d}}a{\rm{d}}\theta}{a^{3}}=C_{\psi,M}\,\left\langle f,g \right\rangle_{L^{2}(\mathbb{R}^{2})}, 
			\end{split}
			\label{26}
		\end{align}
		where $C_{\psi,M}$ is give by (\ref{12}).
	\end{property}
	\begin{proof}
		By Lemma \ref{Lemma3}, we get
		\begin{align}
			\begin{split}
				&\quad \,\,\int_{0}^{2\pi}\!\int_{0}^{+\infty}\!\int_{\mathbb{R}^{2}}W^{M}_{f}(\boldsymbol{b},a,\theta)\overline{W^{M}_{g}(\boldsymbol{b},a,\theta)}\frac{{\rm{d}}\boldsymbol{b}{\rm{d}}a{\rm{d}}\theta}{a^{3}}\\
				&=\int_{0}^{2\pi}\!\int_{0}^{+\infty}\!\int_{\mathbb{R}^{2}}a\!\int_{\mathbb{R}^{2}}e^{j\left[  -\frac{A}{2B}\boldsymbol{b}^{2}-\frac{D}{2B}\boldsymbol{\xi}^{2}+\frac{1}{B}\boldsymbol{b}\boldsymbol{\xi}+\frac{D}{2B}(a R_{-\theta}\boldsymbol{\xi})^{2}\right] }\mathscr{L}^{M}\left\lbrace f\right\rbrace(\boldsymbol{\xi})\,{\rm{d}}\boldsymbol{\xi}\\
				&\quad \times a\,\overline{\mathscr{L}^{M}\!\left\lbrace e^{-j \frac{A}{2B}(\cdot)^{2}}\psi\, \right\rbrace\!(aR_{-\theta}\boldsymbol{\xi})}\int_{\mathbb{R}^{2}}e^{j\left[  \frac{A}{2B}\boldsymbol{b}^{2}+\frac{D}{2B}\boldsymbol{\xi}^{2}-\frac{1}{B}\boldsymbol{b}\boldsymbol{\xi}-\frac{D}{2B}(a R_{-\theta}\boldsymbol{\xi})^{2}\right] }\\
				&\quad\times \mathscr{L}^{M}\!\left\lbrace e^{-j \frac{A}{2B}(\cdot)^{2}}\psi \right\rbrace\!(aR_{-\theta}\boldsymbol{\xi})\overline{\mathscr{L}^{M}\!\left\lbrace g\right\rbrace(\boldsymbol{\xi})}\,{\rm{d}}\boldsymbol{\xi}\,\frac{{\rm{d}}\boldsymbol{b}{\rm{d}}a{\rm{d}}\theta}{a^{3}}\\
				&=\frac{1}{a}\int_{0}^{2\pi}\!\int_{0}^{+\infty}\!\bigg|\mathscr{L}^{M}\left\lbrace e^{-j \frac{A}{2B}(\cdot)^{2}}\psi \right\rbrace (aR_{-\theta}\boldsymbol{\xi})\bigg|^{2}{\rm{d}}a{\rm{d}}\theta\!\int_{\mathbb{R}^{2}}\!\mathscr{L}^{M}\!\left\lbrace f\right\rbrace(\boldsymbol{\xi})\,\overline{\mathscr{L}^{M}\!\left\lbrace g\right\rbrace(\boldsymbol{\xi})}{\rm{d}}\boldsymbol{\xi}\\
				&=C_{\psi,M}\,\left\langle \mathscr{L}^{M}\!\left\lbrace f\right\rbrace(\boldsymbol{\xi})\,,\,\mathscr{L}^{M}\!\left\lbrace g\right\rbrace(\boldsymbol{\xi}) \right\rangle_{L^{2}(\mathbb{R}^{2})}\\
				&=C_{\psi,M}\,\left\langle f,g \right\rangle_{L^{2}(\mathbb{R}^{2})}.
			\end{split}
			\label{27}
		\end{align}
		The proof is completed. 
	\end{proof} 
	\begin{corollary} 
		If $f=g$, Property \ref{property4} becomes
		\begin{align}
			\begin{split}
				\int_{0}^{2\pi}\int_{0}^{+\infty}\int_{\mathbb{R}^{2}}\bigg|W^{M}_{f}(\boldsymbol{b},a,\theta)\bigg|^{2}\frac{{\rm{d}}\boldsymbol{b}{\rm{d}}a{\rm{d}}\theta}{a^{3}}=C_{\psi,M}\left\langle f,g \right\rangle_{L^{2}(\mathbb{R}^{2})}.
			\end{split}
			\label{28}
		\end{align}
	\end{corollary} 
	\begin{property}[Parseval Formula]\label{property5}
		If $\psi\in L^{2}(\mathbb{R}^{2})$ is polar mother wavelet, \textcolor{red}{$f\in L^{2}(\mathbb{R}^{2})$.} Then
		\begin{align}
			\begin{split}
				\int_{\mathbb{R}^{2}}\bigg|\mathscr{L}^{M}\left\lbrace W^{M}_{f}(\boldsymbol{b},a,\theta)\right\rbrace (\boldsymbol{\xi})\bigg|^{2}{\rm{d}}\boldsymbol{\xi}&=\int_{\mathbb{R}^{2}}\bigg|W^{M}_{f}(\boldsymbol{b},a,\theta)\bigg|^{2}{\rm{d}}\boldsymbol{b}=\Vert f\Vert^{2}_{L^{2}(\mathbb{R}^{2})}. 
			\end{split}
			\label{29}
		\end{align}
	\end{property}
	\begin{proof}
		The proof can be easily obtained.
	\end{proof}
	\begin{property}[Energy Preserving Relation]\label{property6}
		If $\psi\in L^{2}(\mathbb{R}^{2})$ is polar mother wavelet, $f\in L^{2}(\mathbb{R}^{2})$. Then
		\begin{align}
			\begin{split}
				\int_{0}^{2\pi}\!\int_{0}^{+\infty}\!\int_{\mathbb{R}^{2}}\bigg|W^{M}_{f}(\boldsymbol{b},a,\theta)\bigg|^{2}\frac{{\rm{d}}\boldsymbol{b}{\rm{d}}a{\rm{d}}\theta}{a^{3}}=C_{\psi,M}\Vert f\Vert^{2}_{L^{2}(\mathbb{R}^{2})}, 
			\end{split}
			\label{30}
		\end{align}
		where $C_{\psi,M}$ is give by (\ref{12}).
	\end{property}
	\begin{proof}
		The proof can be easily deduced by Property \ref{property5}.
	\end{proof}
	\subsection{Inversion Formula}
	The identity (\ref{12}) is called the admissibility condition which guarantees the existence of the inversion formula for the PLCWT.
	\begin{lemma}\label{lemma4}
		If $\psi\in L^{2}(\mathbb{R}^{2})$ is polar mother wavelet. Then, the PLCWT of every $f\in L^{2}(\mathbb{R}^{2})$ can be changed to the PWT, that is
		\begin{align}
			\begin{split}
				W^{M}_{f}(\boldsymbol{b},a,\theta)=e^{-j\frac{A}{2B}\boldsymbol{b}^{2}}W_{\widetilde{f}}\,(\boldsymbol{b},a,\theta), 
			\end{split}
			\label{31}
		\end{align}
		where $\widetilde{f}(\boldsymbol{t})=f(\boldsymbol{t})\cdot e^{j\frac{A}{2B}\boldsymbol{t}^{2}}$.
	\end{lemma}
	\begin{proof}
		The proof can be obtained directly from Formula (\ref{4}).
	\end{proof}
	\begin{lemma}\label{lemma5}
		If $\psi\in L^{2}(\mathbb{R}^{2})$ is polar mother wavelet. Then, for every $f\in L^{2}(\mathbb{R}^{2})$, the LCT of the PLCWT is obtained by
		\begin{align}
			\begin{split}
			 \mathscr{L}^{M}\left\lbrace W^{M}_{f}(\boldsymbol{b},a,\theta)\right\rbrace (\boldsymbol{\xi})=ae^{j\frac{D}{2B}(a R_{-\theta}\boldsymbol{\xi})^{2}}\overline{\mathscr{L}^{M}\left\lbrace e^{-j \frac{A}{2B}(\cdot)^{2}}\psi \right\rbrace (aR_{-\theta}\boldsymbol{\xi})}\mathscr{L}^{M}\left\lbrace f\right\rbrace (\boldsymbol{\xi}). 
			\end{split}
			\label{32}
		\end{align}
	\end{lemma}
	\begin{proof}
		According to the definition of the LCT, we obtain 
		\begin{align}
			\begin{split}
				&\quad \,\mathscr{L}^{M}\left\lbrace W^{M}_{f}(\boldsymbol{b},a,\theta)\right\rbrace (\boldsymbol{\xi})\\
				&=\frac{1}{2\pi B}\int_{\mathbb{R}^{2}}W^{M}_{f}(\boldsymbol{b},a,\theta)e^{j\left( \frac{A}{2B}\boldsymbol{b}^{2}-\frac{1}{B}\boldsymbol{b}\cdot\boldsymbol{\xi}+\frac{D}{2B}\boldsymbol{\xi}^{2}\right) }{\rm{d}}\boldsymbol{b}\\
				&=\frac{1}{2\pi B}\!\int_{\mathbb{R}^{2}}\!\left\lbrace \frac{1}{a}\int_{\mathbb{R}^{2}}\!f(\boldsymbol{t})\overline{\psi\left(\frac{R_{-\theta}\left(\boldsymbol{t}-\boldsymbol{b} \right) }{a} \right)}e^{j\frac{A}{2B}(\boldsymbol{t}^{2}-\boldsymbol{b}^{2})}{\rm{d}}\boldsymbol{t}\right\rbrace\\
				&\quad \times e^{j\left( \frac{A}{2B}\boldsymbol{b}^{2}-\frac{1}{B}\boldsymbol{b}\cdot\boldsymbol{\xi}+\frac{D}{2B}\boldsymbol{\xi}^{2}\right) }{\rm{d}}\boldsymbol{b}\\
				&=\!\frac{1}{2\pi aB}\!\!\int_{\mathbb{R}^{2}}\!\! \int_{\mathbb{R}^{2}}\!\!f(\boldsymbol{t})\overline{\psi\left(\frac{R_{-\theta}\left(\boldsymbol{t}\!-\!\boldsymbol{b}\right) }{a} \right)}\!e^{j\left( \frac{A}{2B}\boldsymbol{t}^{2}-\frac{1}{B}\boldsymbol{b}\cdot\boldsymbol{\xi}+\frac{D}{2B}\boldsymbol{\xi}^{2}\right) }{\rm{d}}\boldsymbol{b}{\rm{d}}\boldsymbol{t}\\
				&=\frac{1}{2\pi B}\!\int_{\mathbb{R}^{2}}f(\boldsymbol{w}+\boldsymbol{b})e^{j\left[  \frac{A}{2B}(\boldsymbol{w}+\boldsymbol{b})^{2}-\frac{1}{B}(\boldsymbol{w}+\boldsymbol{b})\cdot\boldsymbol{\xi}+\frac{D}{2B}\boldsymbol{\xi}^{2}\right] }{\rm{d}}(\boldsymbol{w}+\boldsymbol{b})\\
				&\quad\times \frac{1}{a}\int_{\mathbb{R}^{2}}\overline{\psi\left(\frac{R_{-\theta}\left(\boldsymbol{t}-\boldsymbol{b} \right) }{a} \right)}e^{j\frac{1}{B}\boldsymbol{w}\cdot\boldsymbol{\xi}}{\rm{d}}\boldsymbol{w}\\
				&=ae^{j\frac{D}{2B}(a R_{-\theta}\boldsymbol{\xi})^{2}}\overline{\mathscr{L}^{M}\left\lbrace e^{-j \frac{A}{2B}(\cdot)^{2}}\psi \right\rbrace (aR_{-\theta}\boldsymbol{\xi})}\mathscr{L}^{M}\left\lbrace f\right\rbrace (\boldsymbol{\xi}).
			\end{split}
			\label{33}
		\end{align}
		The proof is completed. 
	\end{proof}
	\begin{theorem}[Inversion Formula]\label{theorem1}
		If $\psi\in L^{2}(\mathbb{R}^{2})$ is polar mother wavelet, $f, g\in L^{2}(\mathbb{R}^{2})$. Then $f$ can be reconstructed by 
		\begin{align}
			\begin{split}
				f(\boldsymbol{t})=\frac{1}{C_{\psi,M}}\int_{0}^{2\pi}\!\int_{0}^{+\infty}\!\int_{\mathbb{R}^{2}}W^{M}_{f}(\boldsymbol{b},a,\theta)\psi^{M}_{\boldsymbol{b},a,\theta}(\boldsymbol{t})\frac{{\rm{d}}\boldsymbol{b}{\rm{d}}a{\rm{d}}\theta}{a^{3}},
			\end{split}
			\label{34}
		\end{align}
		where $C_{\psi,M}$ is the admissibility condition give by (\ref{12}).
	\end{theorem}
	\begin{proof}
		According the LCT, we get 
		\begin{align}
			\begin{split}
				\mathscr{L}^{M}\left\lbrace\overline{\psi_{\boldsymbol{b},a,\theta}^{M}} \right\rbrace(\boldsymbol{\xi})
				&=\frac{1}{2\pi aB}\int_{\!\mathbb{R}^{2}}\!\!\!e^{j\frac{A}{2B}(\boldsymbol{t}^{2}\!-\!\boldsymbol{b}^{2})}\overline{\psi\left(\frac{R_{-\theta}\left(\boldsymbol{t}-\boldsymbol{b} \right) }{a}\! \right)}e^{j\left( \frac{A}{2B}\boldsymbol{b}^{2}\!-\!\frac{1}{B}\boldsymbol{b}\cdot\boldsymbol{\xi}+\frac{D}{2B}\boldsymbol{\xi}^{2}\right) }{\rm{d}}\boldsymbol{b}\\
				&=\frac{1}{2\pi aB}\int_{\mathbb{R}^{2}}\overline{\psi\left(\frac{R_{-\theta}\left(\boldsymbol{t}-\boldsymbol{b} \right) }{a} \right)}e^{j\left[  \frac{A}{2B}\boldsymbol{t}^{2}-\frac{1}{B}\boldsymbol{b}\cdot\boldsymbol{\xi}+\frac{D}{2B}\boldsymbol{\xi}^{2}\right]}{\rm{d}}\boldsymbol{b}.
			\end{split}
			\label{35}
		\end{align}
		Let $\boldsymbol{\eta}=\boldsymbol{t}-\boldsymbol{b}$, (\ref{35}) can be written as
		\begin{align}
			\begin{split}
				&\quad \mathscr{L}^{M}\left\lbrace\overline{\psi_{\boldsymbol{b},a,\theta}^{M}} \right\rbrace(\boldsymbol{\xi})=\frac{1}{2\pi aB}\int_{\mathbb{R}^{2}}\overline{\psi\left(\frac{R_{-\theta}\boldsymbol{\eta} }{a} \right)}e^{j\left[  \frac{A}{2B}\boldsymbol{t}^{2}-\frac{1}{B}\left(\boldsymbol{t}-\boldsymbol{\eta} \right)\cdot\boldsymbol{\xi}+\frac{D}{2B}\boldsymbol{\xi}^{2}\right]}{\rm{d}}\boldsymbol{\eta}.
			\end{split}
			\label{36}
		\end{align}
		Let $\boldsymbol{z}=R_{-\theta}\boldsymbol{\eta}/a$, we have
		\begin{align}
			\begin{split}
				\mathscr{L}^{M}\left\lbrace\overline{\psi_{\boldsymbol{b},a,\theta}^{M}} \right\rbrace(\boldsymbol{\xi})&=ae^{j\left[  \frac{A}{2B}\boldsymbol{t}^{2}+\frac{D}{2B}\boldsymbol{\xi}^{2}-\frac{1}{B}\boldsymbol{t}\cdot\boldsymbol{\xi}+\frac{D}{2B}(a R_{-\theta}\boldsymbol{\xi})^{2}\right] }\overline{\mathscr{L}^{M}\left\lbrace e^{-j \frac{A}{2B}(\cdot)^{2}}\psi \right\rbrace (aR_{-\theta}\boldsymbol{\xi})}.
			\end{split}
			\label{37}
		\end{align}
		Using Lemma \ref{lemma5}, we obtain
		\begin{align}
			\begin{split}
				&\quad \mathscr{L}^{M}\left\lbrace W^{M}_{f}(\boldsymbol{b},a,\theta)\right\rbrace (\boldsymbol{\xi})=ae^{j\frac{D}{2B}(a R_{-\theta}\boldsymbol{\xi})^{2}}\overline{\mathscr{L}^{M}\left\lbrace e^{-j \frac{A}{2B}(\cdot)^{2}}\psi \right\rbrace (aR_{-\theta}\boldsymbol{\xi})}\mathscr{L}^{M}\left\lbrace f\right\rbrace (\boldsymbol{\xi}). 
			\end{split}
			\label{38}
		\end{align}
		From (\ref{37}), (\ref{38}), and admissibility condition, we have
		\begin{align}
			\begin{split}
				&\quad \,\frac{1}{C_{\psi,M}}\int_{0}^{2\pi}\!\int_{0}^{+\infty}\!\int_{\mathbb{R}^{2}}W^{M}_{f}(\boldsymbol{b},a,\theta)\psi^{M}_{\boldsymbol{b},a,\theta}(\boldsymbol{t})\frac{{\rm{d}}\boldsymbol{b}{\rm{d}}a{\rm{d}}\theta}{a^{3}}\\
				&=\!\frac{1}{C_{\psi,M}}\!\!\int_{0}^{2\pi}\!\!\int_{0}^{+\infty}\!\!\left\lbrace \int_{\mathbb{R}^{2}}e^{j\left[  -\frac{A}{2B}\boldsymbol{t}^{2}-\frac{D}{2B}\boldsymbol{\xi}^{2}+\frac{1}{B}\boldsymbol{t}\cdot\boldsymbol{\xi}-\frac{D}{2B}(R_{-\theta}\boldsymbol{\xi})^{2}\right] }\right.\\  
				&\left.\quad\times \mathscr{L}^{M}\left\lbrace e^{-j \frac{A}{2B}(\cdot)^{2}}\psi \right\rbrace (aR_{-\theta}\boldsymbol{\xi}) a^{2}e^{j\frac{D}{2B}(a R_{-\theta}\boldsymbol{\xi})^{2}}\right.\\
				&\left.\quad\times\overline{\mathscr{L}^{M}\left\lbrace e^{-j \frac{A}{2B}(\cdot)^{2}}\psi \right\rbrace (aR_{-\theta}\boldsymbol{\xi})}\mathscr{L}^{M}\left\lbrace f\right\rbrace (\boldsymbol{\xi}){\rm{d}}\boldsymbol{\xi}\right\rbrace\frac{{\rm{d}}a{\rm{d}}\theta}{a^{3}}\\
				&=\frac{1}{C_{\psi,M}}\int_{0}^{\infty}\int_{0}^{2\pi}\bigg|\mathscr{L}^{M}\left\lbrace e^{-j \frac{A}{2B}(\cdot)^{2}}\psi \right\rbrace (aR_{-\theta}\boldsymbol{\xi})\bigg|^{2}\frac{{\rm{d}}a\rm{d}\theta}{a}\\
				&\quad\times \int_{\mathbb{R}^{2}}\mathscr{L}^{M}\left\lbrace f\right\rbrace (\boldsymbol{\xi})e^{j\left[  -\frac{A}{2B}\boldsymbol{t}^{2}-\frac{D}{2B}\boldsymbol{\xi}^{2}+\frac{1}{B}\boldsymbol{t}\cdot\boldsymbol{\xi}\right] }{\rm{d}}\boldsymbol{\xi}\\
				&=f(\boldsymbol{t}).
			\end{split}
			\label{39}
		\end{align}
		The proof is completed. 
	\end{proof}
	\section{Convolution and Correlation Theorems}
	\label{Convolution}
	Convolution theorem is the most fundamental theory in linear time-invariant (LIT) systems, so it is essential to study the convolution and correlation theorems of the PLCWT.
	
	\subsection{Convolution and Correlation Operators}
	First some definitions of convolution and correlation operators are given.
	\begin{definition}[Convolution Operators]\label{definition1}
		Let $f$ and $g$ are complex-valued functions defined on $\mathbb{R}^{2}$, 
		\begin{enumerate}
			\item Let $f\in L^{2}(\mathbb{R}^{2})$ and $g\in L^{2}(\mathbb{R}^{2})$, then the convolution operator $\circledast$ is defined as
			\begin{align}
				\begin{split}
					h(\boldsymbol{x})=\left( f\circledast g\right) (\boldsymbol{x})=\int_{\mathbb{R}^{2}}f(\boldsymbol{t})g(\boldsymbol{x}-\boldsymbol{t}){\rm{d}}\boldsymbol{t}. 
				\end{split}
				\label{40}
			\end{align}
			If $g\in L^{1}(\mathbb{R}^{2})$, then the integral $h$ exists almost everywhere, $f\circledast g\in L^{2}(\mathbb{R}^{2})$, and $\Vert f\circledast g\Vert_{L^{2}}=\Vert f \Vert_{L^{2}}\cdot\Vert g \Vert_{L^{1}}$.
			\item 
			Let $f\in L^{1}(\mathbb{R}^{2})$, $g\in L^{1}(\mathbb{R}^{2})\bigcap L^{2}(\mathbb{R}^{2})$,  it is possible to convolve them in just one variable
			as follows:
			\begin{align}
				\begin{split}
					h_{1}(\boldsymbol{y},\lambda,\theta)&=\left( f\circledast_{1} g\right) (\boldsymbol{y},\lambda,\theta)=\int_{\mathbb{R}^{2}}f(\boldsymbol{t},\lambda,\theta)g(\boldsymbol{y}-\boldsymbol{t},\lambda,\theta){\rm{d}}\boldsymbol{t},
				\end{split}
				\label{41}
			\end{align}
			where $\boldsymbol{y}\in \mathbb{R}^{2}$, $a\in{\mathbb{R}^{+}}$, and $\theta\in\left[0,2\pi \right] $.
			\item Let $\psi_{f}\in L^{1}(\mathbb{R}^{2})$ and $\psi_{g}\in L^{1}(\mathbb{R}^{2})\bigcap L^{2}(\mathbb{R}^{2})$ be two admissible polar mother wavelets, their convolution operation $\circledast^{\psi}$ is defined as a new function
			\begin{align}
				\begin{split}
					\psi_{h}(\boldsymbol{t})&=\left( \psi_{f}\circledast_{\psi} \psi_{g}\right) (\boldsymbol{t})=\int_{\mathbb{R}^{2}}\psi_{f}(R_{-\theta}\boldsymbol{y})\psi_{g}\left( R_{-\theta}\left(R_{\theta}\boldsymbol{t}-\boldsymbol{y} \right) \right){\rm{d}}\boldsymbol{y}. 
				\end{split}
				\label{42}
			\end{align}
		\end{enumerate}
	\end{definition} 
	\begin{definition}[Correlation Operators]\label{definition2} 
		Let $f$ and $g$ are complex-valued functions defined on $\mathbb{R}^{2}$, then as follows
		\begin{enumerate}
			\item Let $f\in L^{2}(\mathbb{R}^{2})$, $g\in L^{1}(\mathbb{R}^{2})\bigcap L^{2}(\mathbb{R}^{2})$, the correlation operator $\odot$ is defined as
			\begin{align}
				\begin{split}
					h(\boldsymbol{x})=\left( f\odot g\right) (\boldsymbol{x})=\int_{\mathbb{R}^{2}}f(\boldsymbol{t})g(\boldsymbol{x}+\boldsymbol{t}){\rm{d}}\boldsymbol{t}. 
				\end{split}
				\label{43}
			\end{align}
			When $f=g$, $f\odot g$ is called the autocorrelation of $f$, and $g$ is called the cross-correlation of $f$ and $g$.
			\item 
			Let $f\in L^{1}(\mathbb{R}^{2})$, $g\in L^{1}(\mathbb{R}^{2})\bigcap L^{2}(\mathbb{R}^{2})$, it is possible to correlate them in just one variable
			as follows
			\begin{align}
				\begin{split}
					h_{1}(\boldsymbol{y},\lambda,\theta)&=\left( f\odot_{1} g\right) (\boldsymbol{y},\lambda,\theta)=\int_{\mathbb{R}^{2}}f(\boldsymbol{t},\lambda,\theta)g(\boldsymbol{y}+\boldsymbol{t},\lambda,\theta){\rm{d}}\boldsymbol{t},
				\end{split}
				\label{44}
			\end{align}
			where $\boldsymbol{y}\in \mathbb{R}^{2}$, $a\in{\mathbb{R}^{+}}$, and $\theta\in\left[0,2\pi \right] $.
			\item Let $\psi_{f}\in L^{1}(\mathbb{R}^{2})$ and $\psi_{g}\in L^{1}(\mathbb{R}^{2})\bigcap L^{2}(\mathbb{R}^{2})$ be two admissible polar mother wavelets, their convolution operation $\odot^{\psi}$ is defined as a new function
			\begin{align}
				\begin{split}
					\psi_{h}(\boldsymbol{t})&=\left( \psi_{f}\odot_{\psi} \psi_{g}\right) (\boldsymbol{t})=\int_{\mathbb{R}^{2}}\psi_{f}(R_{-\theta}\boldsymbol{y})\psi_{g}\left( R_{-\theta}\left(R_{\theta}\boldsymbol{t}+\boldsymbol{y} \right) \right){\rm{d}}\boldsymbol{y}. 
				\end{split}
				\label{45}
			\end{align}
		\end{enumerate}
	\end{definition} 
	Since both the convolution and correlation of two valid wavelets satisfy the necessary admissibility conditions, we can use them to analyze convolution and correlation signals. The proofs of the convolution and correlation theorems of the PLCWT is given in the following part.
	
	\subsection{Convolution Theorem}
	\begin{theorem}[Convolution Theorem] \label{theorem2}
		Let $\psi_{f}\in L^{2}(\mathbb{R}^{2})$ and $\psi_{g}\in L^{1}(\mathbb{R}^{2})\bigcap L^{2}(\mathbb{R}^{2})$ be two admissible polar mother wavelets. If $h=\left( f\circledast g\right)$ and $\psi_{h}=\left( \psi_{f}\circledast_{\psi} \psi_{g}\right)$, Then the PLCWT of $h$ is given by
		\begin{align}
			\begin{split}
				W^{M}_{h}(\boldsymbol{b},a,\theta)=\left(W^{M}_{f}\circledast_{1} W^{M}_{g}\right)(\boldsymbol{b},a,\theta),	 
			\end{split}
			\label{46}
		\end{align}
		where $W^{M}_{f}$ and $W^{M}_{g}$ denote the PLCWT of $f\in L^{2}(\mathbb{R}^{2})$ and $g\in L^{1}(\mathbb{R}^{2})\bigcap L^{2}(\mathbb{R}^{2})$, respectively.
		\begin{proof}
			The PLCWT of $h(\boldsymbol{t})$, given by (\ref{10}), we get
			\begin{align}
				\begin{split}
					W^{M}_{h}(\boldsymbol{b},a,\theta)=\frac{1}{a}\!\int_{\mathbb{R}^{2}}\!h(\boldsymbol{x})\overline{\psi_{h}\left(\frac{R_{-\theta}\left(\boldsymbol{x}\!-\!\boldsymbol{b} \right) }{a} \right)}e^{j\frac{A}{2B}(\boldsymbol{x}^{2}-\boldsymbol{b}^{2})}{\rm{d}}\boldsymbol{x}.
				\end{split}
				\label{47}
			\end{align}
			According to the definition of the convolution operators in Definition \ref{definition1}, we have
			\begin{align}
				\begin{split}
					h(\boldsymbol{x})=\left( f\circledast g\right) (\boldsymbol{x})=\int_{\mathbb{R}^{2}}f(\boldsymbol{t})g(\boldsymbol{x}-\boldsymbol{t}){\rm{d}}\boldsymbol{t},
				\end{split}
				\label{48}
			\end{align}
			\begin{align}
				\begin{split}
					\psi_{h}\!\!\left( \frac{R_{-\theta}\left(\boldsymbol{x}\!-\!\boldsymbol{b} \right) }{a}\right)\!\!=\!\int_{\mathbb{R}^{2}}\!\!\psi_{f}(R_{-\theta}\boldsymbol{y})\psi_{g}\left(\! \frac{\!R_{-\theta}\left(\boldsymbol{x}\!-\!\boldsymbol{b}\!-\!a\boldsymbol{y} \!\right)}{a} \right){\rm{d}}\boldsymbol{y}. 
				\end{split}
				\label{49}
			\end{align}
			Substituting (\ref{48}) and (\ref{49}) into (\ref{47}), (\ref{47}) can be written as
			\begin{align}
				\begin{split}
					W^{M}_{h}(\boldsymbol{b},a,\theta)&=\frac{1}{a}\int_{\mathbb{R}^{2}}h(\boldsymbol{x})\overline{\psi_{h}\left(\frac{R_{-\theta}\left(\boldsymbol{x}-\boldsymbol{b} \right) }{a} \right)}e^{j\frac{A}{2B}(\boldsymbol{x}^{2}-\boldsymbol{b}^{2})}{\rm{d}}\boldsymbol{x}\\
					&=\frac{1}{a}e^{j\frac{A}{2B}(\boldsymbol{x}^{2}-\boldsymbol{b}^{2})}\int_{\mathbb{R}^{2}}\int_{\mathbb{R}^{2}}\int_{\mathbb{R}^{2}}f(\boldsymbol{t})g(\boldsymbol{x}-\boldsymbol{t})\psi_{f}(R_{-\theta}\boldsymbol{y})\\
					&\quad\times\overline{\psi_{g}\left( \frac{R_{-\theta}\left(\boldsymbol{x}-\boldsymbol{b}-a\boldsymbol{y} \right)}{a} \right)}{\rm{d}}\boldsymbol{t}{\rm{d}}\boldsymbol{x}{\rm{d}}\boldsymbol{y}.
				\end{split}
				\label{50}
			\end{align}
			Performing the change of the variables $\boldsymbol{m}=\boldsymbol{x}-\boldsymbol{t}$ and $\boldsymbol{n}=\boldsymbol{b}-a\boldsymbol{y}-\boldsymbol{t}$ satisfy $2\boldsymbol{n}^{2}-2\boldsymbol{b}\cdot\boldsymbol{n}+2\boldsymbol{t}\cdot\boldsymbol{m}=0$, we obtain
			\begin{align}
				\begin{split}
					W^{M}_{h}(\boldsymbol{b},a,\theta)
					&=\frac{1}{a^{2}}\,e^{\,j\frac{A}{2B}\left[ \boldsymbol{t}^{2}-\left(\boldsymbol{b}-\boldsymbol{n} \right)^{2}+\boldsymbol{m}^{2}-\boldsymbol{n}^{2}\right] }\int_{\mathbb{R}^{2}}\,\int_{\mathbb{R}^{2}}\,\int_{\mathbb{R}^{2}}f\left( \boldsymbol{t} \right) g(\boldsymbol{m})\\
					&\quad\times\overline{\psi_{f}\left( \frac{R_{\!-\theta}\left(\boldsymbol{t}\!-\!\left( \boldsymbol{b}\!-\!\boldsymbol{n}\right)\right)}{a}\right)}\overline{\psi_{g}\left( \frac{R_{-\theta}\left(\boldsymbol{m}-\boldsymbol{n}\right)}{a} \right)}{\rm{d}}\boldsymbol{t}{\rm{d}}\boldsymbol{n}{\rm{d}}\boldsymbol{m}\\
					&=\int_{\mathbb{R}^{2}}\frac{1}{a}\int_{\mathbb{R}^{2}}f\left( \boldsymbol{t} \right)\overline{\psi_{f}\left( \frac{R_{\!-\theta}\left(\boldsymbol{t}\!-\!\left( \boldsymbol{b}\!-\!\boldsymbol{n}\right)\right)}{a}\right)}e^{j\frac{A}{2B}\left[ \boldsymbol{t}^{2}-\left(\boldsymbol{b}-\boldsymbol{n} \right)^{2}\right] }{\rm{d}}\boldsymbol{t}\\
					&\quad\times\frac{1}{a}\int_{\mathbb{R}^{2}}g(\boldsymbol{m})\overline{\psi_{g}\left( \frac{R_{-\theta}\left(\boldsymbol{m}-\boldsymbol{n}\right)}{a} \right)}{\rm{d}}\boldsymbol{m}{\rm{d}}\boldsymbol{n}\\
					&=\int_{\mathbb{R}^{2}}W^{M}_{f}(\boldsymbol{b}-\boldsymbol{n},a,\theta)W^{M}_{g}(\boldsymbol{n},a,\theta){\rm{d}}\boldsymbol{n}\\
					&=\left(W^{M}_{f}\circledast_{1} W^{M}_{g}\right)(\boldsymbol{b},a,\theta).
				\end{split}
				\label{51}
			\end{align}
			The proof is completed. 
		\end{proof}
	\end{theorem}
	\begin{corollary}\label{corollary2} 
		It can be observed that for $M=\left [ \begin{matrix}
			0& 1 \\
			-1& 0 \\
		\end{matrix} \right ] $ in Theorem \ref{theorem2}, we have
		\begin{align}
			\begin{split}
				W_{h}(\boldsymbol{b},a,\theta)=\left(W_{f}\circledast_{1} W_{g}\right)(\boldsymbol{b},a,\theta),	 
			\end{split}
			\label{52}
		\end{align}
		where $W_{f}$ and $W_{g}$ denote the PWT of the functions $f$ and $g$ with wavelets $\psi_{f}$ and $\psi_{g}$ respectively.
	\end{corollary}
	\begin{corollary}\label{corollary3} 
		It is interesting to note that by taking $\theta=0$ in Corollary \ref{corollary2}, it degenerates into the classical convolution theorem of the WT \cite{ref41}.
	\end{corollary}
	
	\subsection{Correlation Theorem}
	\begin{theorem}[Correlation theorem]\label{theorem3}
		Let $\psi_{f}\in L^{2}(\mathbb{R}^{2})$ and $\psi_{g}\in L^{1}(\mathbb{R}^{2})\bigcap L^{2}(\mathbb{R}^{2})$ be two admissible polar mother wavelets. If $h=\left( f\odot g\right)$ and $\psi_{h}=\left( \psi_{f}\odot_{\psi} \psi_{g}\right)$, Then the PLCWT of $h$ is given by
		\begin{align}
			\begin{split}
				W^{M}_{h}(\boldsymbol{b},a,\theta)=\left(W^{M}_{f}\odot_{1} W^{M}_{g}\right)(-\boldsymbol{b},a,\theta),	 
			\end{split}
			\label{53}
		\end{align}
		where $W^{M}_{f}$ and $W^{M}_{g}$ denote the PLCWT of the functions $f\in L^{2}(\mathbb{R}^{2})$ and $g\in L^{1}(\mathbb{R}^{2})\bigcap L^{2}(\mathbb{R}^{2})$, respectively.
	\end{theorem}	
	\begin{proof}
		Using the definition of the LCT, we have
		\begin{align}
			\begin{split}
				W^{M}_{h}(\boldsymbol{b},a,\theta)=\frac{1}{a}\int_{\mathbb{R}^{2}}h(\boldsymbol{x})\overline{\psi_{h}\left(\frac{R_{-\theta}\left(\boldsymbol{x}-\boldsymbol{b} \right) }{a} \right)}e^{j\frac{A}{2B}(\boldsymbol{x}^{2}-\boldsymbol{b}^{2})}{\rm{d}}\boldsymbol{x}.
			\end{split}
			\label{54}
		\end{align}
		According to the definition of the correlation operators, we obtain
		\begin{align}
			\begin{split}
				h(\boldsymbol{x})=\left( f\odot g\right) (\boldsymbol{x})=\int_{\mathbb{R}^{2}}f(\boldsymbol{t})g(\boldsymbol{x}+\boldsymbol{t}){\rm{d}}\boldsymbol{t},
			\end{split}
			\label{55}
		\end{align}
		\begin{align}
			\begin{split}
				\psi_{h}\left( \frac{R_{\!-\theta}\left(\boldsymbol{x}\!-\!\boldsymbol{b} \right) }{a}\right)\!=\!\!\int_{\mathbb{R}^{2}}\!\!\psi_{f}(R_{-\theta}\boldsymbol{y})\psi_{g}\left( \frac{R_{\!-\theta}\left(\boldsymbol{x}\!-\!\boldsymbol{b}\!+\!a\boldsymbol{y} \right)}{a} \right)\!{\rm{d}}\boldsymbol{y}. 
			\end{split}
			\label{56}
		\end{align}
		Combining (\ref{54}), (\ref{55}), and (\ref{56}), we obtain
		\begin{align}
			\begin{split}
				W^{M}_{h}(\boldsymbol{b},a,\theta)&=\frac{1}{a}e^{j\frac{A}{2B}(\boldsymbol{x}^{2}-\boldsymbol{b}^{2})}\int_{\mathbb{R}^{2}}\int_{\mathbb{R}^{2}}\int_{\mathbb{R}^{2}}f(\boldsymbol{t})g(\boldsymbol{x}+\boldsymbol{t})\\
				&\quad \times\psi_{f}(R_{\!-\theta}\boldsymbol{y})\overline{\psi_{g}\left( \frac{R_{\!-\theta}\left(\boldsymbol{x}\!-\!\boldsymbol{b}\!+\!a\boldsymbol{y} \right)}{a} \right)}{\rm{d}}\boldsymbol{t}{\rm{d}}\boldsymbol{x}{\rm{d}}\boldsymbol{y}.
			\end{split}
			\label{57}
		\end{align}
		Let variables $\boldsymbol{m}=\boldsymbol{x}+\boldsymbol{t}$ and $\boldsymbol{n}=\boldsymbol{b}-a\boldsymbol{y}+\boldsymbol{t}$ satisfy $2\boldsymbol{n}^{2}-2\boldsymbol{b}\cdot\boldsymbol{n}+2\boldsymbol{t}\cdot\boldsymbol{m}=0$, we have
		\begin{align}
			\begin{split}
				W^{M}_{h}(\boldsymbol{b},a,\theta)
				&=\frac{1}{a^{2}}e^{j\frac{A}{2B}\left[ \boldsymbol{t}^{2}-\left(\boldsymbol{n}-\boldsymbol{b} \right)^{2}+\boldsymbol{m}^{2}-\boldsymbol{n}^{2}\right] }\int_{\mathbb{R}^{2}}\int_{\mathbb{R}^{2}}\int_{\mathbb{R}^{2}}f\left( \boldsymbol{t} \right) g(\boldsymbol{m})\\
				&\quad\overline{\psi_{f}\left( \frac{R_{\!-\theta}\left(\boldsymbol{t}\!-\!\left( \boldsymbol{n}\!-\!\boldsymbol{b}\right)\right)}{a}\right)}\overline{\psi_{g}\left( \frac{R_{\!-\theta}\left(\boldsymbol{m}\!-\!\boldsymbol{n}\right)}{a} \right)}{\rm{d}}\boldsymbol{t}{\rm{d}}\boldsymbol{n}{\rm{d}}\boldsymbol{m}\\
				&=\int_{\mathbb{R}^{2}}W^{M}_{f}(\boldsymbol{n}-\boldsymbol{b},a,\theta)W^{M}_{g}(\boldsymbol{n},a,\theta){\rm{d}}\boldsymbol{n}\\
				&=\left(W^{M}_{f}\circledast_{1} W^{M}_{g}\right)(\boldsymbol{-b},a,\theta).
			\end{split}
			\label{58}
		\end{align}
		The proof is completed. 
	\end{proof}
	\begin{corollary}\label{corollary4} 
		According to (\ref{53}), for $M=\left [ \begin{matrix}
			0& 1 \\
			-1& 0 \\
		\end{matrix} \right ] $, Theorem \ref{theorem3} will be reduced to
		\begin{align}
			\begin{split}
				W_{h}(\boldsymbol{b},a,\theta)=\left(W_{f}\odot_{1} W_{g}\right)(-\boldsymbol{b},a,\theta),	 
			\end{split}
			\label{59}
		\end{align}
		where $W_{f}$ and $W_{g}$ denote the PWT of the functions $f$ and $g$ with wavelets $\psi_{f}$ and $\psi_{g}$, respectively.
	\end{corollary}
	\begin{corollary}\label{corollary5} 
		It is interesting to note that by taking $\theta=0$ in Corollary \ref{corollary4}, it degenerates into the classical correlation theorem of the WT \cite{ref41}.
	\end{corollary}
	
	\section{Uncertainty Principles Associated with PLCWT}
	\label{Uncertainty}
	One of the most important discoveries is the uncertainty principle, which describes different physical phenomena in different fields. At present, various transforms have been studied in detail, and many theoretical results have been obtained \cite{ref4,ref5,ref19,ref22,ref23,ref42,ref43}. But the discovery of the uncertainty principle related to PLCWT is rare. Thus, three kinds of uncertainty principles associated with the PLCWT are mainly investigated in this section.
	\subsection{Heisenberg’s Inequality}
	\begin{theorem}\label{theorem10}
		Let $\psi\in L^{2}(\mathbb{R}^{2})$ be an admissible polar mother wavelet, and $f\in L^{2}(\mathbb{R}^{2})$, then
		\begin{align}
			\begin{split}
				\left\lbrace \int_{0}^{2\pi}\int_{0}^{+\infty}\int_{\mathbb{R}^{2}}\boldsymbol{b}^{2}\bigg|W^{M}_{f}(\boldsymbol{b},a,\theta)\bigg|^{2}\frac{{\rm{d}}\boldsymbol{b}{\rm{d}}a{\rm{d}}\theta}{a^{3}}\right\rbrace^{\frac{1}{2}}&\left\lbrace \int_{\mathbb{R}^{2}}\boldsymbol{\xi}^{2}\bigg|\mathscr{L}^{M}\left\lbrace f\right\rbrace (\boldsymbol{\xi})\bigg|^{2}{\rm{d}}\boldsymbol{\xi}\right\rbrace ^{\frac{1}{2}}\\
				&\geq\frac{B\sqrt{C_{\psi,M}}}{2}\Vert f\Vert^{2}_{L^{2}(\mathbb{R}^{2})}, \end{split}
			\label{60}
		\end{align} 
		where $C_{\psi,M}$ is give by (\ref{12}).
	\end{theorem}	
	\begin{proof}
		Let $f,\psi\in L^{2}(\mathbb{R}^{2})$, then $$\psi\left(\frac{R_{-\theta}\left(\boldsymbol{t}-\boldsymbol{b} \right) }{a} \right)\in L^{2}(\mathbb{R}^{2}),\quad  W^{M}_{f}(\boldsymbol{b},a,\theta)\in L^{2}(\mathbb{R}^{2}).$$
		If $\mathscr{L}^{M}{f}\in L^{2}(\mathbb{R}^{2})$, the Heisenberg's inequality for the LCT is given by \cite{ref19,ref43}
		\begin{align}
			\begin{split}
				\left\lbrace \int_{\mathbb{R}^{2}}\boldsymbol{b}^{2}\bigg|f(\boldsymbol{b})\bigg|^{2}{\rm{d}}\boldsymbol{b}\right\rbrace ^{\frac{1}{2}}&\left\lbrace \int_{\mathbb{R}^{2}}\boldsymbol{\xi}^{2}\bigg|\mathscr{L}^{M}\left\lbrace f\right\rbrace (\boldsymbol{\xi})\bigg|^{2}{\rm{d}}\boldsymbol{\xi}\right\rbrace ^{\frac{1}{2}}\geq\frac{B}{2}\int_{\mathbb{R}^{2}}\bigg|f(\boldsymbol{b})\bigg|^{2}{\rm{d}}\boldsymbol{b}.  
			\end{split}
			\label{61}
		\end{align} 
		Replacing $f(\boldsymbol{b})$ with $W^{M}_{f}(\boldsymbol{b},a,\theta)$ in (\ref{61}), we obtain 
		\begin{align}
			\begin{split}
				\left\lbrace \int_{\mathbb{R}^{2}}\boldsymbol{b}^{2}\bigg|W^{M}_{f}(\boldsymbol{b},a,\theta)\bigg|^{2}{\rm{d}}\boldsymbol{b}\right\rbrace ^{\frac{1}{2}}
				&\left\lbrace \int_{\mathbb{R}^{2}}\boldsymbol{\xi}^{2}\bigg|\mathscr{L}^{M}\left\lbrace W^{M}_{f}(\boldsymbol{b},a,\theta)\right\rbrace (\boldsymbol{\xi})\bigg|^{2}{\rm{d}}\boldsymbol{\xi}
				\right\rbrace ^{\frac{1}{2}}\\ &\qquad\qquad \geq\frac{B}{2}\int_{\mathbb{R}^{2}}\bigg|W^{M}_{f}(\boldsymbol{b},a,\theta)\bigg|^{2}{\rm{d}}\boldsymbol{b}.  
			\end{split}
			\label{62}
		\end{align} 
		Integrating both sides of (\ref{62}) with respect to the measure ${\rm{d}}\boldsymbol{a}{\rm{d}}\boldsymbol{\theta}/a^{3}$, we have
		\begin{align}
			\begin{split}
				\int_{0}^{2\pi}\!\!\int_{0}^{+\infty}\!\!\left\lbrace \int_{\mathbb{R}^{2}}\boldsymbol{b}^{2}\bigg|W^{M}_{f}(\boldsymbol{b},a,\theta)\bigg|^{2}{\rm{d}}\boldsymbol{b}\right\rbrace ^{\frac{1}{2}}&\left\lbrace \int_{\mathbb{R}^{2}}\!\!\boldsymbol{\xi}^{2}\bigg|\mathscr{L}^{M}\left\lbrace W^{M}_{f}(\boldsymbol{b},a,\theta)\right\rbrace (\boldsymbol{\xi})\bigg|^{2}\!\!{\rm{d}}\boldsymbol{\xi}
				\right\rbrace ^{\frac{1}{2}}\!\!\frac{{\rm{d}}\boldsymbol{a}{\rm{d}}\boldsymbol{\theta}}{a^{3}}\\ 
				&\geq\frac{B}{2}\int_{0}^{2\pi}\!\!\int_{0}^{+\infty}\!\!\int_{\mathbb{R}^{2}}\bigg|W^{M}_{f}(\boldsymbol{b},a,\theta)\bigg|^{2}\frac{{\rm{d}}\boldsymbol{a}{\rm{d}}\boldsymbol{\theta}{\rm{d}}\boldsymbol{b}}{a^{3}}.  
			\end{split}
			\label{63}
		\end{align} 
		By Schwartz's inequality and using Property \ref{property6}, we obtain
		\begin{align}
			\begin{split}
				&\left\lbrace \int_{0}^{2\pi}\!\!\int_{0}^{+\infty}\!\!\int_{\mathbb{R}^{2}}\!\!\boldsymbol{b}^{2}\bigg|W^{M}_{f}(\boldsymbol{b},a,\theta)\bigg|^{2}\frac{{\rm{d}}\boldsymbol{b}{\rm{d}}\boldsymbol{a}{\rm{d}}\boldsymbol{\theta}}{a^{3}}\right\rbrace ^{\frac{1}{2}}\\
				&\quad\times\left\lbrace \!\!\int_{0}^{2\pi}\!\!\int_{0}^{+\infty}\!\!\int_{\mathbb{R}^{2}}\boldsymbol{\xi}^{2}\bigg|\mathscr{L}^{M}\left\lbrace W^{M}_{f}(\boldsymbol{b},a,\theta)\right\rbrace (\boldsymbol{\xi})\bigg|^{2}\frac{{\rm{d}}\boldsymbol{\xi}{\rm{d}}\boldsymbol{a}{\rm{d}}\boldsymbol{\theta}}{a^{3}}\right\rbrace^{\frac{1}{2}}\\
				&\quad \geq\frac{B}{2}\!\!\int_{0}^{2\pi}\!\!\int_{0}^{+\infty}\!\!\int_{\mathbb{R}^{2}}\bigg|W^{M}_{f}(\boldsymbol{b},a,\theta)\bigg|^{2}\frac{{\rm{d}}\boldsymbol{b}{\rm{d}}\boldsymbol{a}{\rm{d}}\boldsymbol{\theta}}{a^{3}}\\
				&\quad=\frac{B}{2}C_{\psi,M}\Vert f\Vert^{2}_{L^{2}(\mathbb{R}^{2})}.  
			\end{split}
			\label{64}
		\end{align} 
		Using Lemma \ref{lemma5} and the admissibility condition, we get
		\begin{align}
			\begin{split}
				&\quad \,\,\int_{0}^{2\pi}\int_{0}^{+\infty}\int_{\mathbb{R}^{2}}\boldsymbol{\xi}^{2}\bigg|\mathscr{L}^{M}\left\lbrace W^{M}_{f}(\boldsymbol{b},a,\theta)\right\rbrace (\boldsymbol{\xi})\bigg|^{2}\frac{{\rm{d}}\boldsymbol{\xi}{\rm{d}}\boldsymbol{a}{\rm{d}}\boldsymbol{\theta}}{a^{3}}\\
				&=\int_{0}^{2\pi}\!\!\int_{0}^{+\infty}\!\!\int_{\mathbb{R}^{2}}\bigg|ae^{j\frac{D}{2B}(a R_{\!-\theta}\boldsymbol{\xi})^{2}}\overline{\mathscr{L}^{M}\left\lbrace e^{-j \frac{A}{2B}(\cdot)^{2}}\psi \right\rbrace\! (aR_{-\theta}\boldsymbol{\xi})}\\
				&\qquad\times\mathscr{L}^{M}\left\lbrace f\right\rbrace (\boldsymbol{\xi})\bigg|^{2}\boldsymbol{\xi}^{2}\,\frac{{\rm{d}}\boldsymbol{\xi}{\rm{d}}\boldsymbol{a}{\rm{d}}\boldsymbol{\theta}}{a^{3}}\\
				&=C_{\psi,M}\int_{0}^{+\infty}\boldsymbol{\xi}^{2}\bigg|\mathscr{L}^{M}\left\lbrace f\right\rbrace (\boldsymbol{\xi})\bigg|^{2}{\rm{d}}\boldsymbol{\xi}.
			\end{split}
			\label{65}
		\end{align} 
		Substituting (\ref{65}) into (\ref{64}), (\ref{64}) can be written as
		\begin{align}
			\begin{split}
				\left\lbrace \int_{0}^{2\pi}\int_{0}^{+\infty}\int_{\mathbb{R}^{2}}\boldsymbol{b}^{2}\bigg|W^{M}_{f}(\boldsymbol{b},a,\theta)\bigg|^{2}\frac{{\rm{d}}\boldsymbol{b}{\rm{d}}a{\rm{d}}\theta}{a^{3}}\right\rbrace^{\frac{1}{2}}&\left\lbrace \int_{\mathbb{R}^{2}}\boldsymbol{\xi}^{2}\bigg|\mathscr{L}^{M}\left\lbrace f\right\rbrace (\boldsymbol{\xi})\bigg|^{2}{\rm{d}}\boldsymbol{\xi}\right\rbrace ^{\frac{1}{2}}\\
				&\quad\geq\frac{B\sqrt{C_{\psi,M}}}{2}\Vert f\Vert^{2}_{L^{2}(\mathbb{R}^{2})}.  
			\end{split}
			\label{66}
		\end{align} 
		The proof is completed.
	\end{proof}
	\begin{corollary}\label{corollary6} 
		It can be observed that for $M=\left [ \begin{matrix}
			0& 1 \\
			-1& 0 \\
		\end{matrix} \right ]$, Theorem \ref{theorem10} will be reduced to
		\begin{align}
			\begin{split}
				&\left\lbrace \int_{0}^{2\pi}\!\!\int_{0}^{+\infty}\!\!\int_{\mathbb{R}^{2}}\boldsymbol{b}^{2}\bigg|W_{f}(\boldsymbol{b},a,\theta)\bigg|^{2}\frac{{\rm{d}}\boldsymbol{b}{\rm{d}}a{\rm{d}}\theta}{a^{3}}\right\rbrace^{\frac{1}{2}}\left\lbrace \int_{\mathbb{R}^{2}}\boldsymbol{\xi}^{2}\bigg|\hat{f} (\boldsymbol{\xi})\bigg|^{2}{\rm{d}}\boldsymbol{\xi}\right\rbrace ^{\frac{1}{2}}\!\!\geq\!\!\frac{\sqrt{C_{\psi}}}{2}\Vert f\Vert^{2}_{L^{2}(\mathbb{R}^{2})},  
			\end{split}
			\label{601}
		\end{align} 
		where $C_{\psi}$ is the admissibility condition of the PWT \cite{ref23}, $W_{f}$ denotes the PWT of the functions $f$, and $\hat{f}$ denotes the FT of the functions $f$. 
	\end{corollary}
	\subsection{Generalized Heisenberg’s Uncertainty Principle}
	A generalization of the Heisenberg’s inequality for the PLCWT is presented in the following theorem.
	\begin{theorem}\label{theorem11}
		Let $\psi\in L^{2}(\mathbb{R}^{2})$ be an admissible polar mother wavelet and $1\leq p\leq2$. Then, for any $f\in L^{2}(\mathbb{R}^{2})$, we have
		\begin{align}
			\begin{split}
				&\left\lbrace \int_{0}^{2\pi}\!\!\int_{0}^{+\infty}\!\!\int_{\mathbb{R}^{2}}\boldsymbol{b}^{p}\bigg|W^{M}_{f}(\boldsymbol{b},a,\theta)\bigg|^{p}\frac{{\rm{d}}\boldsymbol{b}{\rm{d}}a{\rm{d}}\theta}{a^{3}}\right\rbrace^{\frac{1}{p}}\left\lbrace \int_{\mathbb{R}^{2}}\boldsymbol{\xi}^{p}\bigg|\mathscr{L}^{M}\left\lbrace f\right\rbrace (\boldsymbol{\xi})\bigg|^{p}{\rm{d}}\boldsymbol{\xi}
				\right\rbrace ^{\frac{1}{p}}\\
				&\qquad\qquad\qquad\qquad\qquad\qquad\qquad\qquad\geq\frac{a^{\frac{3}{p}-\frac{3}{2}}B^{\frac{1}{p}+\frac{1}{2}}}{2}\sqrt{C_{\psi,M}}\Vert f\Vert^{2}_{L^{2}(\mathbb{R}^{2})},  
			\end{split}
			\label{67}
		\end{align} 
		where $C_{\psi,M}$ is give by (\ref{12}).
	\end{theorem}	
	\begin{proof}
		The generalization of the Heisenberg's inequality associated with the LCT is given by \cite{ref19,ref43}
		\begin{align}
			\begin{split}
				\left\lbrace \int_{\mathbb{R}^{2}}\boldsymbol{b}^{p}\bigg|f(\boldsymbol{b})\bigg|^{p}{\rm{d}}\boldsymbol{b}\right\rbrace ^{\frac{1}{p}}&\left\lbrace \int_{\mathbb{R}^{2}}\boldsymbol{\xi}^{p}\bigg|\mathscr{L}^{M}\left\lbrace f\right\rbrace (\boldsymbol{\xi})\bigg|^{p}{\rm{d}}\boldsymbol{\xi}\right\rbrace ^{\frac{1}{p}}\geq\frac{B^{\frac{1}{p}+\frac{1}{2}}}{2}\int_{\mathbb{R}^{2}}\bigg|f(\boldsymbol{b})\bigg|^{2}{\rm{d}}\boldsymbol{b},  
			\end{split}
			\label{68}
		\end{align} 
		where $1\leq p\leq2$.\\
		Substituting $f(\boldsymbol{b})$ with $W^{M}_{f}(\boldsymbol{b},a,\theta)$ and integrating with respect to the measure ${\rm{d}}\boldsymbol{a}{\rm{d}}\boldsymbol{\theta}/a^{3}$, we have
		\begin{align}
			\begin{split}
				&\left\lbrace \int_{0}^{2\pi}\!\!\int_{0}^{+\infty}\!\!\int_{\mathbb{R}^{2}}\!\!\boldsymbol{b}^{p}\bigg|W^{M}_{f}(\boldsymbol{b},a,\theta)\bigg|^{p}\frac{{\rm{d}}\boldsymbol{b}{\rm{d}}\boldsymbol{a}{\rm{d}}\boldsymbol{\theta}}{a^{3}}\right\rbrace ^{\frac{1}{p}}\\
				&\qquad \left\lbrace \!\!\int_{0}^{2\pi}\!\!\int_{0}^{+\infty}\!\!\int_{\mathbb{R}^{2}}\boldsymbol{\xi}^{p}\bigg|\mathscr{L}^{M}\left\lbrace W^{M}_{f}(\boldsymbol{b},a,\theta)\right\rbrace (\boldsymbol{\xi})\bigg|^{p}\frac{{\rm{d}}\boldsymbol{\xi}{\rm{d}}\boldsymbol{a}{\rm{d}}\boldsymbol{\theta}}{a^{3}}\right\rbrace^{\frac{1}{p}}\\
				&\qquad\qquad\qquad\geq\frac{B^{\frac{1}{p}+\frac{1}{2}}}{2}C_{\psi,M}\Vert f\Vert^{2}_{L^{2}(\mathbb{R}^{2})}.
			\end{split}
			\label{69}
		\end{align} 
		By Lemma \ref{lemma5}, we get
		\begin{align}
			\begin{split}
				&\quad \,\, \int_{0}^{2\pi}\int_{0}^{+\infty}\int_{\mathbb{R}^{2}}\boldsymbol{\xi}^{p}\bigg|\mathscr{L}^{M}\left\lbrace W^{M}_{f}(\boldsymbol{b},a,\theta)\right\rbrace (\boldsymbol{\xi})\bigg|^{p}\frac{{\rm{d}}\boldsymbol{\xi}{\rm{d}}\boldsymbol{a}{\rm{d}}\boldsymbol{\theta}}{a^{3}}\\
				&=\int_{0}^{2\pi}\int_{0}^{+\infty}\int_{\mathbb{R}^{2}}\boldsymbol{\xi}^{p}\bigg|ae^{j\frac{D}{2B}(a R_{-\theta}\boldsymbol{\xi})^{2}}\mathscr{L}^{M}\left\lbrace f\right\rbrace (\boldsymbol{\xi})\\
				&\quad \times\mathscr{L}^{M}\left\lbrace e^{-j \frac{A}{2B}(\cdot)^{2}}\psi \right\rbrace (aR_{-\theta}\boldsymbol{\xi})\bigg|^{p}\frac{{\rm{d}}\boldsymbol{\xi}{\rm{d}}\boldsymbol{a}{\rm{d}}\boldsymbol{\theta}}{a^{3}}\\
				&=a^{p-3}\int_{0}^{2\pi}\int_{0}^{+\infty}\int_{\mathbb{R}^{2}}\boldsymbol{\xi}^{p}\bigg|\mathscr{L}^{M}\left\lbrace e^{-j \frac{A}{2B}(\cdot)^{2}}\psi \right\rbrace (aR_{-\theta}\boldsymbol{\xi})\bigg|^{p}\\
				&\quad \times\bigg|\mathscr{L}^{M}\left\lbrace f\right\rbrace (\boldsymbol{\xi})\bigg|^{p}{\rm{d}}\boldsymbol{\xi}{\rm{d}}\boldsymbol{a}{\rm{d}}\boldsymbol{\theta}.
			\end{split}
			\label{70}
		\end{align} 
		According to (\ref{12}), (\ref{70}) can be written as
		\begin{align}
			\begin{split}
				&\quad \int_{0}^{2\pi}\int_{0}^{+\infty}\int_{\mathbb{R}^{2}}\boldsymbol{\xi}^{p}\bigg|\mathscr{L}^{M}\left\lbrace W^{M}_{f}(\boldsymbol{b},a,\theta)\right\rbrace (\boldsymbol{\xi})\bigg|^{p}\frac{{\rm{d}}\boldsymbol{\xi}{\rm{d}}\boldsymbol{a}{\rm{d}}\boldsymbol{\theta}}{a^{3}}\\
				&\leq a^{\frac{3p}{2}-3}\left\lbrace \int_{0}^{2\pi}\!\!\int_{0}^{+\infty}\!\!\bigg|\mathscr{L}^{M}\!\left\lbrace e^{-j \frac{A}{2B}(\cdot)^{2}}\psi \right\rbrace\! (aR_{\!-\theta}\boldsymbol{\xi})\bigg|^{p}\frac{{\rm{d}}\boldsymbol{a}{\rm{d}}\boldsymbol{\theta}}{a}\right\rbrace ^{\frac{p}{2}}\\
				&\quad\times\left\lbrace \int_{\mathbb{R}^{2}}\boldsymbol{\xi}^{p}\bigg|\mathscr{L}^{M}\left\lbrace f\right\rbrace (\boldsymbol{\xi})\bigg|^{p}{\rm{d}}\boldsymbol{\xi}\right\rbrace\\
				&=a^{\frac{3p}{2}-3}C_{\psi,M}^{\frac{p}{2}}\int_{0}^{+\infty}\boldsymbol{\xi}^{p}\bigg|\mathscr{L}^{M}\left\lbrace f\right\rbrace (\boldsymbol{\xi})\bigg|^{p}{\rm{d}}\boldsymbol{\xi}.
			\end{split}
			\label{71}
		\end{align} 
		Multiplying (\ref{69}), (\ref{70}), (\ref{71}), and Schwartz's inequality, we get
		\begin{align}
			\begin{split}
				&\quad \left\lbrace \int_{0}^{2\pi}\!\!\int_{0}^{+\infty}\!\!\int_{\mathbb{R}^{2}}\!\!\boldsymbol{b}^{p}\bigg|W^{M}_{f}(\boldsymbol{b},a,\theta)\bigg|^{p}\frac{{\rm{d}}\boldsymbol{b}{\rm{d}}a{\rm{d}}\theta}{a^{3}}\right\rbrace^{\frac{1}{p}}\left\lbrace \int_{\mathbb{R}^{2}}\boldsymbol{\xi}^{p}\bigg|\mathscr{L}^{M}\left\lbrace f\right\rbrace (\boldsymbol{\xi})\bigg|^{p}{\rm{d}}\boldsymbol{\xi}
				\right\rbrace ^{\frac{1}{p}}\\
				&\geq \frac{a^{-\frac{3}{2}+\frac{3}{p}}}{\sqrt{C_{\psi,M}}}\left\lbrace \int_{0}^{2\pi}\int_{0}^{+\infty}\int_{\mathbb{R}^{2}}\boldsymbol{b}^{p}\bigg|W^{M}_{f}(\boldsymbol{b},a,\theta)\bigg|^{p}\frac{{\rm{d}}\boldsymbol{b}{\rm{d}}a{\rm{d}}\theta}{a^{3}}\right\rbrace^{\frac{1}{p}}\\
				&\quad\times \left\lbrace\!\int_{0}^{2\pi}\!\!\!\int_{0}^{+\infty}\!\!\!\int_{\mathbb{R}^{2}}\!\!\boldsymbol{\xi}^{p}\!\bigg|\mathscr{L}^{M}\left\lbrace W^{M}_{f}(\boldsymbol{b},a,\theta)\right\rbrace (\boldsymbol{\xi})\bigg|^{p}\frac{{\rm{d}}\boldsymbol{\xi}{\rm{d}}\boldsymbol{a}{\rm{d}}\boldsymbol{\theta}}{a^{3}} \!\right\rbrace ^{\frac{1}{p}}\\
				&\geq\frac{a^{-\frac{3}{2}+\frac{3}{p}}}{\sqrt{C_{\psi,M}}}\int_{0}^{2\pi}\int_{0}^{+\infty}\left\lbrace \int_{\mathbb{R}^{2}}\boldsymbol{b}^{p}\bigg|W^{M}_{f}(\boldsymbol{b},a,\theta)\bigg|^{p}{\rm{d}}\boldsymbol{b}\right\rbrace^{\frac{1}{p}}\\
				&\quad\times\left\lbrace\int_{\mathbb{R}^{2}}\boldsymbol{\xi}^{p}\bigg|\mathscr{L}^{M}\left\lbrace W^{M}_{f}(\boldsymbol{b},a,\theta)\right\rbrace (\boldsymbol{\xi})\bigg|^{p}{\rm{d}}\boldsymbol{\xi} \right\rbrace ^{\frac{1}{p}}\frac{{\rm{d}}\boldsymbol{a}{\rm{d}}\boldsymbol{\theta}}{a^{3}}\\
				&\geq \frac{a^{-\frac{3}{2}+\frac{3}{p}}}{\sqrt{C_{\psi,M}}}\cdot \frac{B^{\frac{1}{p}+\frac{1}{2}}}{2}C_{\psi,M}\Vert f\Vert^{2}_{L^{2}(\mathbb{R}^{2})}\\
				&=\frac{a^{\frac{3}{p}-\frac{3}{2}}B^{\frac{1}{p}+\frac{1}{2}}}{2}\sqrt{C_{\psi,M}}\Vert f\Vert^{2}_{L^{2}(\mathbb{R}^{2})}.
			\end{split}
			\label{72}
		\end{align} 
		The proof is completed.
	\end{proof}
	\begin{corollary}\label{corollary7} 
		When $p=2$, Theorem \ref{theorem11} boils down to Theorem \ref{theorem10}. This proves that Theorem \ref{theorem11} is a generalization of Theorem \ref{theorem10}.
	\end{corollary}
	\subsection{Logarithmic Uncertainty Principle}
	Here, we derive the logarithmic uncertainty principle for the PLCWT as follow.
	\begin{theorem}\label{theorem12}
		Let $\psi\in L^{2}(\mathbb{R}^{2})$ be an admissible polar mother wavelet. Then, for any $f\in L^{2}(\mathbb{R}^{2})$ such that $W^{M}_{f}(\boldsymbol{b},a,\theta)\in L^{2}(\mathbb{R}^{2})$, we have
		\begin{align}
			\begin{split}
				&\int_{0}^{2\pi}\int_{0}^{+\infty}\int_{\mathbb{R}^{2}}\ln|\boldsymbol{b}|\,\bigg|W^{M}_{f}(\boldsymbol{b},a,\theta)\bigg|^{2}\frac{{\rm{d}}\boldsymbol{b}{\rm{d}}a{\rm{d}}\theta}{a^{3}}+\int_{\mathbb{R}^{2}}\ln|\boldsymbol{b}|\,\bigg|\mathscr{L}^{M}\left\lbrace f\right\rbrace (\boldsymbol{\xi})\bigg|^{p}{\rm{d}}\boldsymbol{\xi}\\ &\qquad\qquad\qquad\qquad\qquad\qquad\qquad\qquad\qquad \geq \left(\mu+\ln B \right) C_{\psi,M}\Vert f\Vert^{2}_{L^{2}(\mathbb{R}^{2})},
			\end{split}
			\label{73}
		\end{align} 
		where $\mu=\chi(\frac{1}{2})-\ln\pi$, $\chi(z)=\frac{\rm{d}}{{\rm{d}}z}\ln\left[ \Gamma(z)\right]$, $\Gamma(\cdot)$ is the gamma function, and $C_{\psi,M}$ is give by (\ref{12}).
	\end{theorem}
	\begin{proof}
		By the logarithmic uncertainty uncertainty for the LCT \cite{ref19,ref43}, we get
		\begin{align}
			\begin{split}
				&\int_{\mathbb{R}^{2}}\ln|\boldsymbol{b}|\,\bigg|f(\boldsymbol{b})\bigg|^{2}{\rm{d}}\boldsymbol{b}+ \int_{\mathbb{R}^{2}}\ln|\boldsymbol{\xi}|\,\bigg|\mathscr{L}^{M}\left\lbrace f\right\rbrace (\boldsymbol{\xi})\bigg|^{2}{\rm{d}}\boldsymbol{\xi}\geq\left(\mu+\ln B \right)\int_{\mathbb{R}^{2}}\bigg|f(\boldsymbol{b})\bigg|^{2}{\rm{d}}\boldsymbol{b}.\end{split}
			\label{74}
		\end{align} 
		Replacing $f$ with $W_{f}^{M}$ of (\ref{74}), we obtain 
		\begin{align}
			\begin{split}
				&\int_{\mathbb{R}^{2}}\ln|\boldsymbol{b}|\,\bigg|W^{M}_{f}(\boldsymbol{b},a,\theta)\bigg|^{2}{\rm{d}}\boldsymbol{b}+\int_{\mathbb{R}^{2}}\ln|\boldsymbol{\xi}|\,\bigg|\mathscr{L}^{M}\left\lbrace W^{M}_{f}(\boldsymbol{b},a,\theta)\right\rbrace (\boldsymbol{\xi})\bigg|^{2}{\rm{d}}\boldsymbol{\xi}\\
				&\qquad\qquad\qquad\qquad\qquad\qquad\quad\, \geq\left(\mu+\ln B \right)\int_{\mathbb{R}^{2}}\bigg|W^{M}_{f}(\boldsymbol{b},a,\theta)\bigg|^{2}{\rm{d}}\boldsymbol{b}.  
			\end{split}
			\label{75}
		\end{align} 
		Integrating the above inequality, we have
		\begin{align}
			\begin{split}
				&\quad\,\int_{0}^{2\pi}\int_{0}^{+\infty}\int_{\mathbb{R}^{2}}\!\!\ln|\boldsymbol{b}|\,\bigg|W^{M}_{f}(\boldsymbol{b},a,\theta)\bigg|^{2}\frac{{\rm{d}}\boldsymbol{b}{\rm{d}}a{\rm{d}}\theta}{a^{3}}\\
				&+\int_{0}^{2\pi}\int_{0}^{+\infty}\int_{\mathbb{R}^{2}}\!\!\ln|\boldsymbol{\xi}|\,\bigg|\mathscr{L}^{M}\left\lbrace W^{M}_{f}(\boldsymbol{b},a,\theta)\right\rbrace (\boldsymbol{\xi})\bigg|^{2}\frac{{\rm{d}}\boldsymbol{\xi}{\rm{d}}a{\rm{d}}\theta}{a^{3}}\\
				& \geq\left(\mu+\ln B \right)\int_{0}^{2\pi}\int_{0}^{+\infty}\int_{\mathbb{R}^{2}}\bigg|W^{M}_{f}(\boldsymbol{b},a,\theta)\bigg|^{2}\frac{{\rm{d}}\boldsymbol{\xi}{\rm{d}}a{\rm{d}}\theta}{a^{3}}\\
				&=\left(\mu+\ln B \right)C_{\psi,M}\Vert f\Vert^{2}_{L^{2}(\mathbb{R}^{2})}.  
			\end{split}
			\label{76}
		\end{align} 
		The second integral formula in (\ref{76}) can be written as
		\begin{align}
			\begin{split}
				&\quad \,\,\int_{0}^{2\pi}\int_{0}^{+\infty}\int_{\mathbb{R}^{2}}\!\!\ln|\boldsymbol{\xi}|\,\bigg|\mathscr{L}^{M}\left\lbrace W^{M}_{f}(\boldsymbol{b},a,\theta)\right\rbrace (\boldsymbol{\xi})\bigg|^{2}\frac{{\rm{d}}\boldsymbol{\xi}{\rm{d}}a{\rm{d}}\theta}{a^{3}} \\
				&=\int_{0}^{2\pi}\int_{0}^{+\infty}\int_{\mathbb{R}^{2}}\ln|\boldsymbol{\xi}|\,\bigg|\mathscr{L}^{M}\left\lbrace W^{M}_{f}(\boldsymbol{b},a,\theta)\right\rbrace (\boldsymbol{\xi})\bigg|\bigg|\overline{\mathscr{L}^{M}\left\lbrace W^{M}_{f}(\boldsymbol{b},a,\theta)\right\rbrace (\boldsymbol{\xi})}\bigg|\frac{{\rm{d}}\boldsymbol{\xi}{\rm{d}}a{\rm{d}}\theta}{a^{3}} \\
				&=\int_{0}^{2\pi}\int_{0}^{+\infty}\int_{\mathbb{R}^{2}}\ln|\boldsymbol{\xi}|\,\frac{{\rm{d}}\,\boldsymbol{\xi}\,{\rm{d}}\,a\,{\rm{d}}\,\theta}{a^{3}}\left\lbrace ae^{j\frac{D}{2B}(a R_{\!-\theta}\boldsymbol{\xi})^{2}}\!\overline{\mathscr{L}^{M}\!\left\lbrace e^{-j \frac{A}{2B}(\cdot)^{2}}\psi \right\rbrace\! (aR_{\!-\theta}\boldsymbol{\xi})}\mathscr{L}^{M}\!\left\lbrace f\right\rbrace\! (\boldsymbol{\xi}) \right\rbrace\\
				&\quad\times\left\lbrace \overline{ae^{j\frac{D}{2B}(a R_{\!-\theta}\boldsymbol{\xi})^{2}}\!\overline{\mathscr{L}^{M}\!\left\lbrace e^{-j \frac{A}{2B}(\cdot)^{2}}\psi \right\rbrace\! (aR_{\!-\theta}\boldsymbol{\xi})}\mathscr{L}^{M}\left\lbrace f\right\rbrace\! (\boldsymbol{\xi})} \right\rbrace \\
				&=\int_{0}^{2\pi}\!\!\int_{0}^{+\infty}\bigg|\mathscr{L}^{M}\left\lbrace e^{-j \frac{A}{2B}(\cdot)^{2}}\psi \right\rbrace (aR_{-\theta}\boldsymbol{\xi})\bigg|^{2}\frac{{\rm{d}}a{\rm{d}}\theta}{a}\int_{\mathbb{R}^{2}}\ln|\boldsymbol{\xi}|\,\bigg|\mathscr{L}^{M}\left\lbrace f\right\rbrace (\boldsymbol{\xi})\bigg|^{2}{\rm{d}}\boldsymbol{\xi}\\
				&=C_{\psi,M}\int_{\mathbb{R}^{2}}\ln|\boldsymbol{\xi}|\,\bigg|\mathscr{L}^{M}\left\lbrace f\right\rbrace (\boldsymbol{\xi})\bigg|^{2}{\rm{d}}\boldsymbol{\xi}.
			\end{split}
			\label{77}
		\end{align} 
		Using (\ref{75}) and (\ref{77}), we have
		\begin{align}
			\begin{split}
				&\int_{0}^{2\pi}\int_{0}^{+\infty}\int_{\mathbb{R}^{2}}\ln|\boldsymbol{b}|\,\bigg|W^{M}_{f}(\boldsymbol{b},a,\theta)\bigg|^{2}\frac{{\rm{d}}\boldsymbol{b}{\rm{d}}a{\rm{d}}\theta}{a^{3}}+\int_{\mathbb{R}^{2}}\ln|\boldsymbol{b}|\,\bigg|\mathscr{L}^{M}\left\lbrace f\right\rbrace (\boldsymbol{\xi})\bigg|^{p}{\rm{d}}\boldsymbol{\xi} \\
				&\qquad\qquad\qquad\qquad\qquad\qquad\qquad\qquad\quad \geq \left(\mu+\ln B \right) C_{\psi,M}\Vert f\Vert^{2}_{L^{2}(\mathbb{R}^{2})}.
			\end{split}
			\label{78}
		\end{align} 
		The proof is completed.
	\end{proof}
	\begin{corollary}\label{corollary8} 
		When $M=\left [ \begin{matrix}
			0& 1 \\
			-1& 0 \\
		\end{matrix} \right ]$, Theorem \ref{theorem12} becomes the logarithmic uncertainty of the PWT \cite{ref19,ref43}.
	\end{corollary}
	
	\section{Potential Application of the PLCWT}
	\label{Applications}
	Image edge detection technology is the most basic content in image processing. Edge is the most important feature of an image. This section focuses on the potential applications of the PLCWT in edge detection. Simulation results show that this method has a good application prospect in edge detection.
	
	\subsection{Image Edge Detection}
	\label{Image}
	The two-dimensional WT can be used for image processing tasks, such as detection, extraction or classification of various features in an image. If the object to be detected has a specific direction, then we need a wavelet with good directional selectivity. This subsection mainly presents the theoretical implementation of the proposed PLCWT method in image edge detection.
	
	Let $f(x,y)$ is a two-dimensional image, $\boldsymbol{b}=(b_{1},b_{2})$. The direct discretization of the PLCWT of $f(x,y)$ is 
	\begin{align}
		\begin{split}
			W^{M}_{f}(\boldsymbol{b},a,\theta_{i})&=\frac{1}{a}\sum_{x=b_{1}-\frac{N}{2}}^{b_{1}+\frac{N}{2}}\sum_{y=b_{2}-\frac{N}{2}}^{b_{2}+\frac{N}{2}}f(x,y)\Phi(x,y),
		\end{split}
		\label{79}
	\end{align}
	in which
	\begin{align}
		\begin{split}
			\Phi(x,y)\!=\!\frac{1}{a}\,e^{-j\!\frac{A}{2B}\left[\left( x^{2}+y^{2}\right) -\left(b_{1}^{2}+b_{2}^{2}\right)  \right] }\psi\left(\!\frac{R_{i}\left(x-b_{1},y-b_{2}\right) }{a} \right),
		\end{split}
		\label{80}
	\end{align}
	where $N$ is the rang of convolution, $R_{i}$ is the rotation of  $\theta_{i}$ \
	
	The potential application of the PLCWT in image edge detection mainly has the following steps:
	\begin{enumerate}[(1)]
		\item Discretization of the PLCWT, as in (\ref{79}). The mother wavelet chooses Morlet wavelet; $\theta_{i}$ is used to detect edge patches almost parallel to the chosen direction, where $\theta_{i}=0, \frac{\pi}{8}, \frac{\pi}{4}, \frac{3\pi}{8}, \frac{\pi}{2}, \frac{5\pi}{8}, \frac{3\pi}{4}$, and $\frac{7\pi}{8}$.
		\item Calculate the PLCWT coefficients, and get the edges of eight directions at different scales. The eight sets of edge segments corresponding to the eight selected rotation angles at different scales can complement each other to form an uninterrupted side.
		\item Perform image fusion on the results in 2), and obtain an edge detection image by thresholding. In practical applications, it is very important to choose the optimal threshold method, and the most suitable threshold method needs to be selected according to the actual situation. 
	\end{enumerate}
	
	Based on the above discussion, we get the theoretical implementation process of the proposed method on edge detection, the schematic diagram is shown in Fig. \ref{fig:2}. In the next subsection, we verify the feasibility of the proposed method through simulation experiments.\
	
	\begin{figure}[t!]
		\centering
		\includegraphics[width=0.8\linewidth]{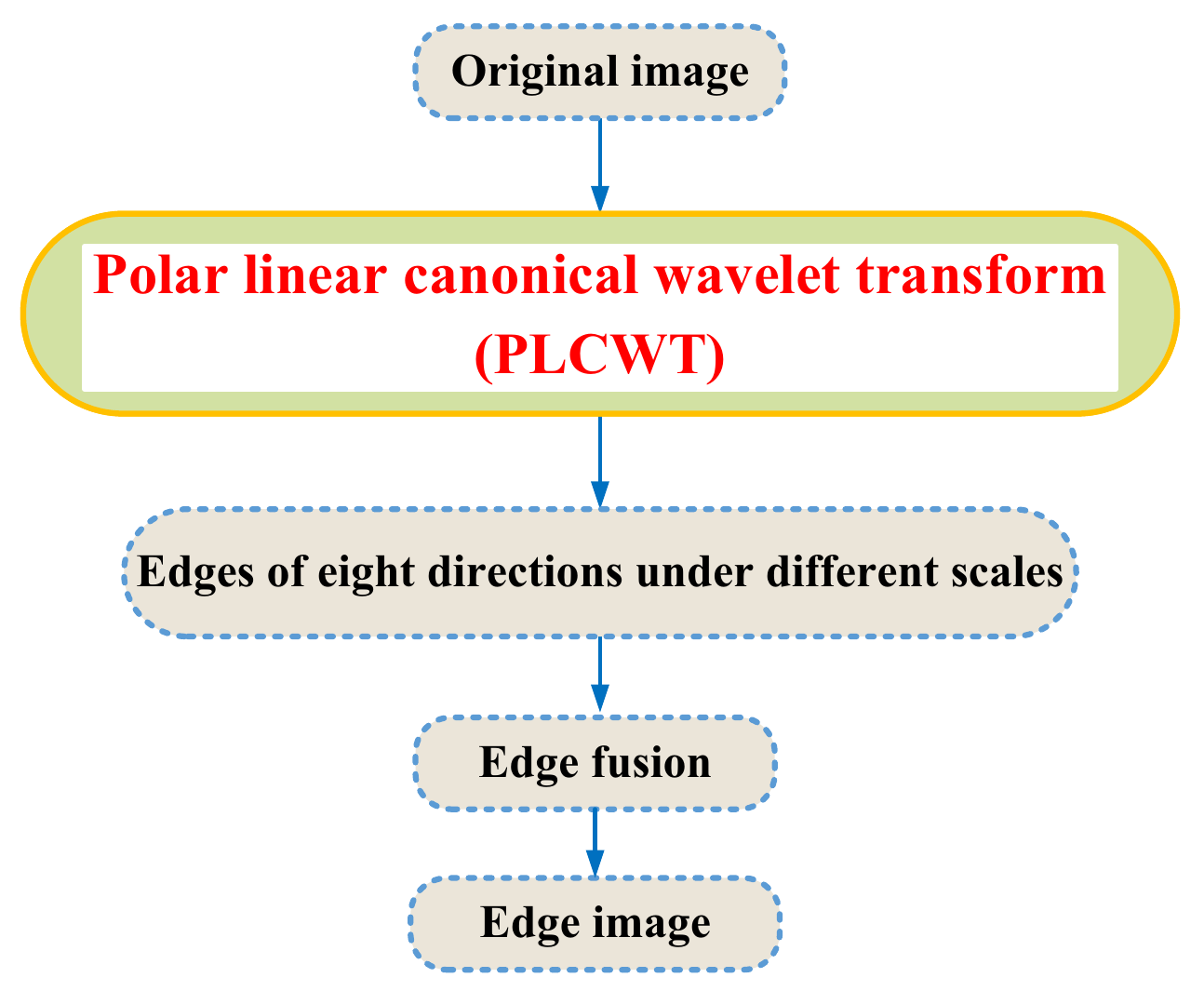}\\
		\caption{Schematic diagram of edge detection based on the proposed method.}
		\label{fig:2} 
	\end{figure}
	
	\subsection{Results of Experiments}
	\label{Results}
	In this subsection, we present some simulation results for different cases and compare the proposed method with the PWT \cite{ref23,ref26} in image edge detection. The simulation results show that the proposed PLCWT is feasible. \
	
	The following three test images are shown in Fig. \ref{fig:7}, which are regular pattern images, portrait images, and landscape images.
	
	\begin{figure}[h!] 
		\centering
		\subfigure[Wheel]{
			\begin{minipage}[b]{.3\linewidth}
				\centering
				\includegraphics[scale=0.43]{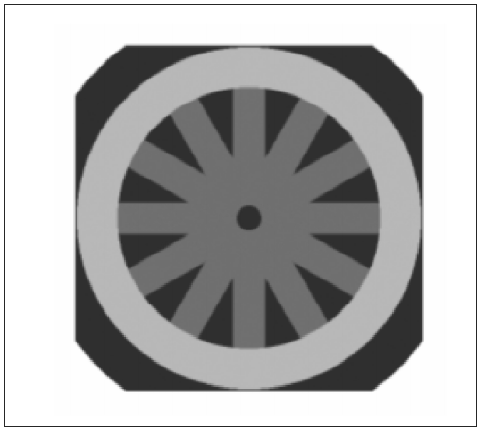}
		\end{minipage}	}
		\subfigure[Lenna]{
			\begin{minipage}[b]{.3\linewidth}
				\centering
				\includegraphics[scale=1.37]{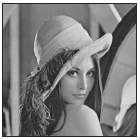}
		\end{minipage}	}
		\subfigure[Walkbridge]{
			\begin{minipage}[b]{.3\linewidth}
				\centering
				\includegraphics[scale=0.7]{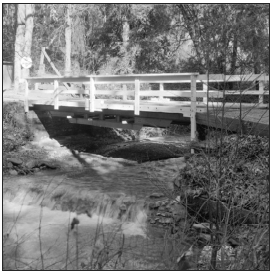}
		\end{minipage}	}
		\caption{The original test images.} 
		\label{fig:7}    
	\end{figure}

	\textbf{Case 1:} Let the parameter $\frac{A}{B}=0.01$, and the scale factor $a=1.5$. We selected 5 scale numbers and 8 directions, where $\theta_{i}=0, \frac{\pi}{8}, \frac{\pi}{4}, \frac{3\pi}{8}, \frac{\pi}{2}, \frac{5\pi}{8}, \frac{3\pi}{4}$, and $\frac{7\pi}{8}$. The edges in eight directions at five different scales extracted from the wheel image (see Fig. \ref{fig:4} (a)) by the polar linear canonical wavelet are shown in Fig. \ref{fig:3}. It can be seen from the Fig. \ref{fig:3} that the polar linear canonical wavelet can effectively extract edges with different directions and is more sensitive to directions. When $\frac{A}{B}=0$, the proposed method degenerates into the PWT. According to the results in Fig. \ref{fig:4}, compared with the PWT \cite{ref23,ref26}, the proposed method in this paper is feasible and effective in image edge detection.
	
	\begin{figure*}[t!]
		\centering
		\includegraphics[width=1.0\linewidth]{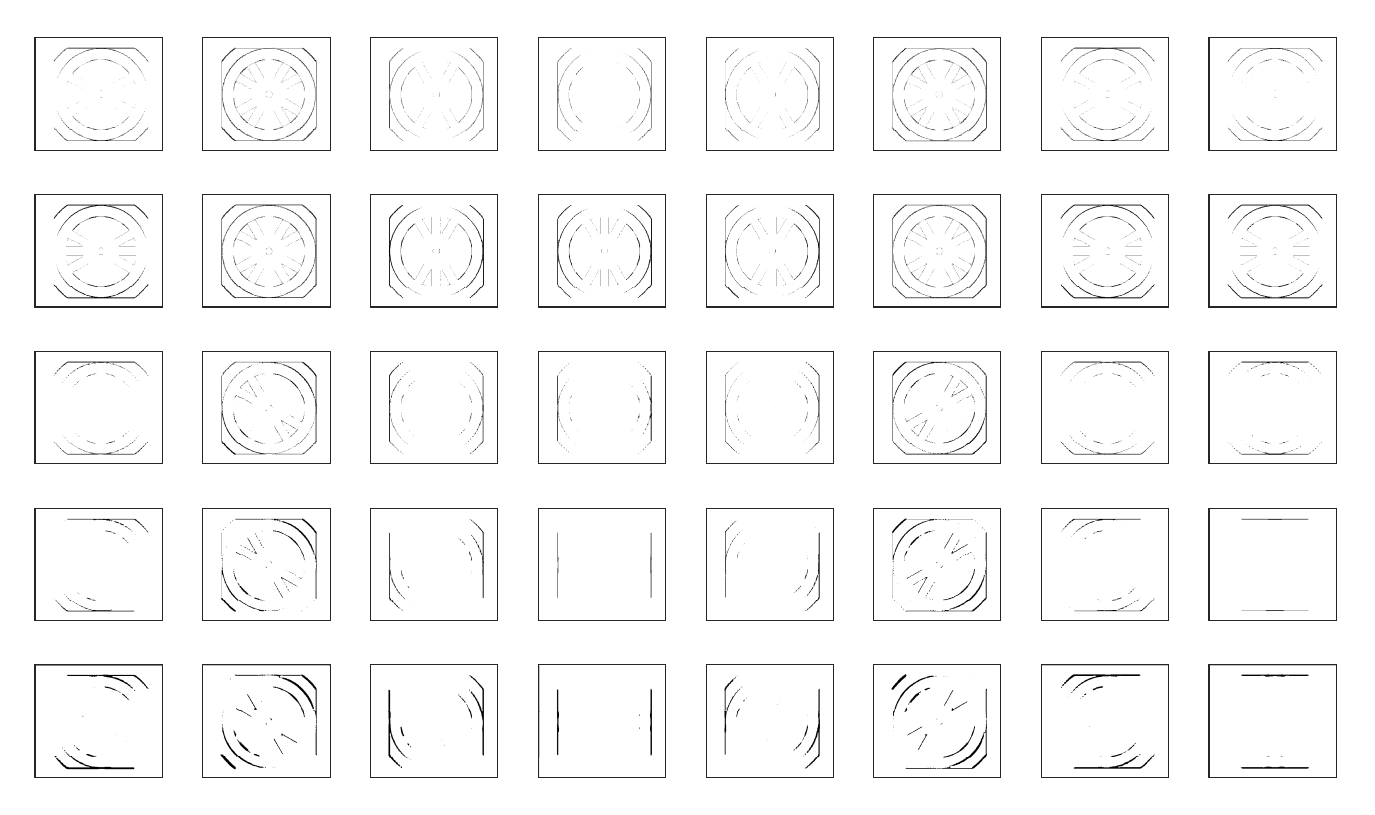}\\
		\caption{Edges of eight directions at different scales. From left to right the direction is $0, \frac{\pi}{8}, \frac{\pi}{4}, \frac{3\pi}{8}, \frac{\pi}{2}, \frac{5\pi}{8}, \frac{3\pi}{4}$, and $\frac{7\pi}{8}$ .}
		\label{fig:3} 
	\end{figure*}
	\begin{figure*}[h!] 
		\centering
		\subfigure[Original wheel image]{
			\begin{minipage}[b]{.3\linewidth}
				\centering
				\includegraphics[scale=0.45]{original_wheel.eps}
		\end{minipage}	}
		\subfigure[The proposed method]{
			\begin{minipage}[b]{.3\linewidth}
				\centering
				\includegraphics[scale=0.45]{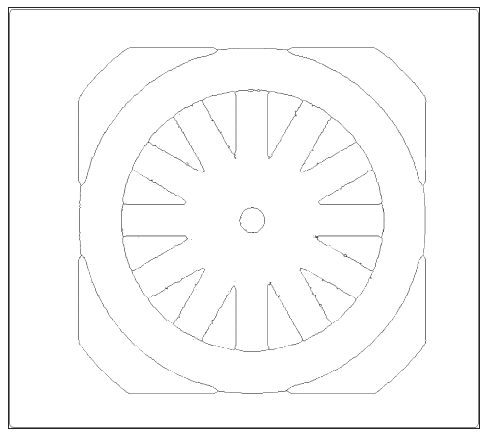}
		\end{minipage}	}
		\subfigure[PWT]{
			\begin{minipage}[b]{.3\linewidth}
				\centering
				\includegraphics[scale=0.45]{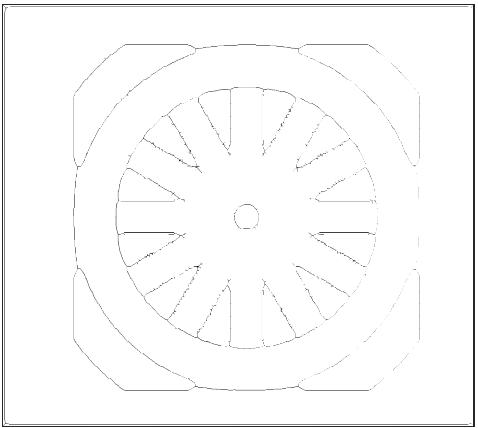}
		\end{minipage}	}
		\caption{Comparison of edge detection of the proposed method and the PWT \cite{ref23,ref26}.}
		\label{fig:4}     
	\end{figure*}
	
	\textbf{Case 2:} Take Lena in Fig. \ref{fig:5} (a) as an test example, the image size is $256*256$, the scale factor $a=1.5$, and the parameters $\frac{A}{B}=0.2,\, 0.12,\, 0.06,\, 0.01,\, 0$. The edge detection of the PLCWT under different parameters is shown in Fig. \ref{fig:5}. Simulation results show that the PLCWT has greater freedom and flexibility in edge detection. The proposed method in this paper can provide more accurate edge direction information, but it will detect many false edges at the same time, and the edge positioning accuracy is not high.
	
	\begin{figure*}[h!] 
		\centering
		\subfigure[Original Lenna image]{
			\begin{minipage}[b]{.3\linewidth}
				\centering
				\includegraphics[scale=1.5]{original_lena.eps}
		\end{minipage}	}
		\subfigure[$\frac{A}{B}=0.2$]{
			\begin{minipage}[b]{.3\linewidth}
				\centering
				\includegraphics[scale=1.5]{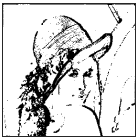}
		\end{minipage}	}
		\subfigure[$\frac{A}{B}=0.12$]{
			\begin{minipage}[b]{.3\linewidth}
				\centering
				\includegraphics[scale=1.5]{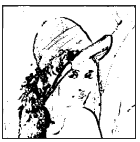}
		\end{minipage}	}
		\subfigure[$\frac{A}{B}=0.06$]{
			\begin{minipage}[b]{.3\linewidth}
				\centering
				\includegraphics[scale=1.5]{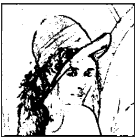}
		\end{minipage}	}
		\subfigure[$\frac{A}{B}=0.01$]{
			\begin{minipage}[b]{.3\linewidth}
				\centering
				\includegraphics[scale=1.5]{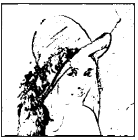}
		\end{minipage}	}
		\subfigure[$\frac{A}{B}=0$]{
			\begin{minipage}[b]{.3\linewidth}
				\centering
				\includegraphics[scale=1.5]{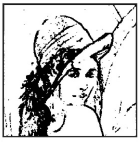}
		\end{minipage}	}
		\caption{The edges detection of the proposed method under different parameters $\frac{A}{B}$} 
		\label{fig:5}    
	\end{figure*}
	
	\textbf{Case 3:} Let the parameter $\frac{A}{B}=0.01$, and the scale factor $a=1.5$. In order to observe the applicability of the proposed method in practical applications, this experiment selects a relatively complex walkbridge image (see Fig. \ref{fig:6} (a)) as the test image, and the image size is $256*256$. Fig. \ref{fig:6} (b) shows that the proposed method can detect more edge details with high localization accuracy. It can be seen from the comparison in Fig. \ref{fig:6} that the method proposed in this paper also has a better effect in detecting the edges of some complex images compared with the PWT \cite{ref23,ref26}.
	
	\begin{figure*}[h!] 
		\centering
		\subfigure[Original walkbridge image]{
			\begin{minipage}[b]{.3\linewidth}
				\centering
				\includegraphics[scale=0.75]{original_bri.eps}
		\end{minipage}	}
		\subfigure[The proposed method]{
			\begin{minipage}[b]{.3\linewidth}
				\centering
				\includegraphics[scale=0.75]{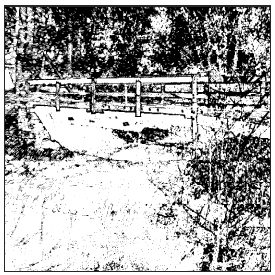}
		\end{minipage}	}
		\subfigure[PWT]{
			\begin{minipage}[b]{.3\linewidth}
				\centering
				\includegraphics[scale=0.75]{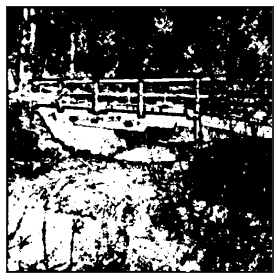}
		\end{minipage}	}
		\caption{Comparison of edge detection of the proposed method and the PWT \cite{ref23,ref26}.} 
		\label{fig:6}   
	\end{figure*}
	
	From the above simulation results, it can be seen that the method proposed in this paper can detect the edge of two-dimensional images, and has a certain prospect in the practical application scenarios related to images.
	
	\section{Conclusion}
	\label{Con}
	In this paper, a new transform method is proposed, namely the PLCWT. Firstly, the basic theory of the PLCWT is given in detail, including its definition, basic properties and inversion Formula. Secondly, the convolution and correlation theorems of the PLCWT are derived. Further, three uncertainty principles related to the PLCWT are also explored. Finally, the simulation experiment of the PLCWT in edge detection is shown, and the results illustrate the correctness and effectiveness of the proposed method.

\end{document}